\documentclass{article}

\usepackage[english]{babel}
\usepackage{graphicx}
\usepackage{textcomp}
\usepackage{amssymb}
\usepackage{amsmath}
\usepackage{setspace}
\usepackage[margin=2.5cm]{geometry}
\usepackage{bm}
\usepackage{centernot}

\newtheorem{proposition}{Proposition}
\newtheorem{corollary}{Corollary}
\newtheorem{remark}{Remark}

\newtheorem{lemma}{Lemma}

\newenvironment{proof}[1][Proof:]{\begin{trivlist}
\item[\hskip \labelsep {\bfseries #1}]}{\end{trivlist}}

\newcommand{\id}{{\mathbf 1}}

\newcommand{\Hol}{\operatorname{Hol}}

\newcommand{\supp}{\operatorname{supp}}

\newcommand{\diag}{\operatorname{diag}}

\newcommand{\Hom}{\operatorname{Hom}}

\newcommand{\Aut}{\mathrm{Aut}}

\newcommand{\Inn}{\mathrm{Inn}}

\newcommand{\ad}{\mbox{\normalfont ad}}

\newcommand{\End}{\mbox{\normalfont End}}

\newcommand{\tr}{\operatorname{tr}}

\newcommand{\aA}{{{\scriptstyle{}^{\scriptstyle a}}\!\!\mathcal{A}}}

\newcommand{\Dil}{\mbox{\normalfont Dil}}

\newcommand{\NM}{{N\!M}}

\newcommand{\bbar}[1]{\overline{\overline{#1}}}

\newcommand{\rank}{{\mathrm{rank}\,}}

\makeatletter
\renewcommand*\env@matrix[1][*\c@MaxMatrixCols c]{%
  \hskip -\arraycolsep
  \let\@ifnextchar\new@ifnextchar
  \array{#1}}
\makeatother

\title{Almost Abelian Lie groups, subgroups and quotients}
\author{Marcelo Almora Rios,\,
Zhirayr Avetisyan,\,
Katalin Berlow,\,
Isaac Martin,\\
Gautam Rakholia,\,
Kelley Yang,\,
Hanwen Zhang,\,
Zishuo Zhao
}

\date{}

\begin{document}

\maketitle

\quad

\begin{center}\textit{Dedicated to the memory of Prof. Nikolai Karapetyants on the occasion of his 80th birthday.}\end{center}

\quad

\abstract{An almost Abelian Lie group is a non-Abelian Lie group with a codimension 1 Abelian normal subgroup. The majority of 3-dimensional real Lie groups are almost Abelian, and they appear in all parts of physics that deal with anisotropic media - cosmology, crystallography etc. In theoretical physics and differential geometry, almost Abelian Lie groups and their homogeneous spaces provide some of the simplest solvmanifolds on which a variety of geometric structures such as symplectic, K\"ahler, spin etc., are currently studied in explicit terms. Recently, almost Abelian Lie algebras were classified and studied in details. However, a systematic investigation of almost Abelian Lie groups has not been carried out yet, and the present paper is devoted to an explicit description of properties of this wide and diverse class of groups.

The subject of investigation are real almost Abelian Lie groups with their Lie group theoretical aspects, such as the exponential map, faithful matrix representations, discrete and connected subgroups, quotients and automorphisms. The emphasis is put on explicit description of all technical details.}

\quad

\quad

\pagebreak

%\documentclass{article}

%\input{../include.tex}

%\begin{document}

\section{Introduction}\label{Intro}

In the present paper we consider only real Lie groups and Lie algebras. An almost Abelian Lie algebra is a non-Abelian Lie algebra $\mathbf{L}$ that contains a codimension one Abelian ideal, and an almost Abelian group is a Lie group with an almost Abelian Lie algebra. This is equivalent to demanding that the Lie group contains a codimension one Abelian normal subgroup. In fact, it can be shown that the existence of a codimension one Abelian Lie subgroup already guarantees the existence of a codimension one Abelian normal subgroup.

In low dimensions almost Abelian Lie groups are well-represented. The only 2-dimensional non-Abelian Lie group is almost Abelian, while 6 out of 9 classes (after Bianchi) of 3-dimensional real Lie algebras/groups are almost Abelian. At the same time, since most physical systems are $n=1,2,3$ dimensional, in absence of rotational symmetries a homogeneous (anisotropic) system is described by a low dimensional Lie group. It is thus only natural that almost Abelian Lie groups are widely used in cosmology, where they represent the symmetries of the universe at large scale (\cite{ElEl98}, \cite{Osi73}, \cite{Pet59}, \cite{Rya75}, \cite{AvVe13} and many others), or crystallography, where they model the symmetries of an ideal solid (\cite{Par16} and references therein). As far as applications in pure mathematics are concerned, one particular almost Abelian Lie group is distinguished - the 3-dimensional Heisenberg group (higher dimensional Heisenberg groups are not almost Abelian). Thorough studies of the Heisenberg group can be found, for instance, in \cite{Fol89} and \cite{Tha98}. Taking roots in the foundations of quantum mechanics, this group has become the classical setting for non-commutative analysis. We refer to \cite{FiRu16} for recently developed theory of quantization and pseudodifferential calculus on the Heisenberg group (among other nilpotent groups). It is therefore desirable to try and extend these results to general almost Abelian groups, but that has to wait until a comprehensive study of almost Abelian Lie groups is available.

Higher dimensional almost Abelian Lie groups have gained in popularity in the last two decades, with at least a dozen papers dealing with the subject written in the last two years only. One context of interest is compact solvmanifolds. A solvmanifold is a homogeneous space $G/N$ with $G$ a simply connected solvable Lie group and $N\subset G$ a discrete subgroup. Almost Abelian groups $G$ are special in that, together with nilpotent groups, these are the only solvable Lie groups for which there is a practically useful necessary and sufficient condition for the solvmanifold $G/N$ to be compact \cite{Boc16}. More generally, almost Abelian groups are unique in their explicit tractability combined with diversity of properties they can possess. A plethora of work in differential geometry and theoretical physics has been devoted to various geometrical constructions on almost Abelian solvmanifolds such as symplectic, K\"ahler, spin, $\mathrm{G}_2$ or $\mathrm{SU}(3)$ structures, various flows etc. \cite{Fre12}, \cite{AnOr17}, \cite{AGMP11}, \cite{CoMa12}, \cite{LaWi19},\cite{FSW19}, \cite{Par21}, \cite{Sta20}, \cite{FiPa20a}, \cite{FiPa20}, \cite{BDV19}, \cite{BaFi18}, \cite{FrSw18}, \cite{BDV18}. In spite of this wide spectrum of interest and applications, to the date there is no comprehensive study of almost Abelian Lie groups in the literature. In the recent papers \cite{Avetisyan2016} and \cite{Avetisyan2018} almost Abelian Lie algebras were studied and their structure was explicitly described. The next step is the study of almost Abelian groups from the Lie group theory perspective, which the present work is mainly devoted to. The far-reaching objective of studying almost Abelian groups systematically is building a variety of well-understood ``lighthouses'' in the sea of solvable Lie groups, as representative as possible, in order to facilitate the development of methods and tools applicable to a wide range of groups.

The following results are obtained in this paper. Let $G$ stand for an almost Abelian group. The exponential map $\exp$ on a simply connected $G$ is described explicitly, and two conditions are given which are equivalent to the injectivity of $\exp$ (exponentiality of $G$). Two faithful matrix representations are introduced for simply connected $G$, and the centre $\mathrm{Z}(G)$ is described. The full automorphism group $\Aut(G)$ and the inner automorphism group $\Inn(G)$ are given explicitly for a connected $G$. Discrete normal subgroups of a simply connected $G$ are studied, and conditions are found for two discrete normal subgroups to be related by an automorphism of $G$. This provides a necessary and sufficient condition for two connected $G$ with the same Lie algebra to be isomorphic, and thus a full classification of connected almost Abelian groups. A necessary and sufficient condition is found for a connected $G$ to admit a faithful matrix representation, and one such representation is given explicitly whenever such exists. Connected subgroups $H\subset G$ of connected $G$ are described, and a condition is established that is equivalent to the closedness of $H$ in $G$.
%\end{document}

\quad

%\documentclass{article}

%\input{include.tex}

%\begin{document}

\section*{Matrix representations of simply connected almost Abelian groups}\label{MatRepSCG}
%\subsubsection*{Z. Avetisyan, K. Berlow, G. Rakholia, Z. Zhao}

A real finite-dimensional almost Abelian Lie algebra is a semidirect product $\mathbb{R}^d\rtimes\mathbb{R}$, and is completely determined by the operator $\ad_{e_0}\in\End(\mathbb{R}^d)$, where $e_0=(0,1)\in\mathbb{R}^d\rtimes\mathbb{R}$. It was shown in \cite{Avetisyan2018} that every such algebra is isomorphic to a representative $\aA_\mathbb{R}(\aleph)$ for which $\ad_{e_0}=\operatorname{J}(\aleph)$, where $\operatorname{J}(\aleph)$ is a (real) Jordan canonical form with its spectral structure encoded in the so-called multiplicity function $\aleph:\mathbb{C}\times\mathbb{N}\to\mathbb{N}_0$ (here $\sigma_\mathbb{R}\simeq\mathbb{C}$ is the set of monic irreducible polynomials over $\mathbb{R}$ as identified with the corresponding roots). So let us consider the real finite-dimensional almost Abelian Lie algebra $\aA(\aleph)=\aA_\mathbb{R}(\aleph)$ corresponding to a finite dimensional multiplicity function $\aleph$ given in its faithful matrix representation
\begin{equation}
\aA(\aleph)\simeq\mathbb{R}^d\rtimes\mathbb{R}\ni(v,t)\quad\mapsto\quad\begin{pmatrix}
0 & 0\\
v & t\operatorname{J}(\aleph)
\end{pmatrix}\in\End(\mathbb{R}^{d+1}).\label{aAMatRep}
\end{equation}
It is straightforward to check that
\begin{equation}
\aA(\aleph)\ni\begin{pmatrix}
0 & 0\\
v & t\operatorname{J}(\aleph)
\end{pmatrix}\mapsto\begin{pmatrix}
0 & 0 & 0\\
v & t\operatorname{J}(\aleph) & 0\\
0 & 0 & t
\end{pmatrix}\label{LieAlgFaithRep2}
\end{equation}
and
\begin{equation}
\aA(\aleph)\ni\begin{pmatrix}
0 & 0\\
v & t\operatorname{J}(\aleph)
\end{pmatrix}\mapsto\begin{pmatrix}
0 & 0 & 0\\
v & t\operatorname{J}(\aleph) & 0\\
t & 0 & 0
\end{pmatrix}\label{LieAlgFaithRep3}
\end{equation}
give further faithful matrix representations of $\aA(\aleph)$. In this section we will establish faithful matrix representations for the simply connected almost Abelian groups corresponding to the Lie algebras $\aA(\aleph)$.

Following \cite{Avetisyan2018}, we identify $\sigma_\mathbb{R}$ with $\mathbb{C}$ in the following way. Every monic irreducible polynomial $p\in\sigma_\mathbb{R}$ is identified with its unique real root $x_p$ if it is first order, and with one of the conjugate pair of complex roots $x_p$ and $\bar{x}_p$ (say, the one in the upper half-plane) if it is second order. Denote
\begin{equation}
T_\aleph\doteq\left\{t\in\mathbb{R}\,{\big|}\quad e^{t\operatorname{J}(\aleph)}=\id\right\}\subset\mathbb{R},\label{TalephDef}
\end{equation}
$$
\mathcal{X}_\aleph\doteq\left\{\omega\in\mathbb{R}\,{\big|}\quad\supp\aleph\subset\imath\omega\mathbb{Z}\right\}.
$$
In the next section we will obtain an explicit description of the set $T_\aleph$. At this point we are mainly interested in matrix representations.

\begin{proposition}\label{NaiveMatRepProp} For a finite multiplicity function $\aleph$ let
$$
G\doteq\left\{\begin{pmatrix}
1 & 0\\
v & e^{t\operatorname{J}(\aleph)}
\end{pmatrix}\,\vline\quad(v,t)\in\mathbb{R}^d\oplus\mathbb{R}\right\}.
$$
Then $G$ is a connected Lie group with Lie algebra $\aA(\aleph)$, and it is simply connected if and only if $T_\aleph=\{0\}$.
\end{proposition}
\begin{proof} That $G$ is a connected Lie group is clear from the definition. For $\forall (u,s)\in\mathbb{R}^d\oplus\mathbb{R}$ let
$$
(-1,1)\ni\tau\mapsto\begin{pmatrix}
1 & 0\\
v(\tau) & e^{t(\tau)\operatorname{J}(\aleph)}
\end{pmatrix}\in G
$$
be a smooth curve with
$$
(v(0),t(0))=(0,0),\quad (v'(0),t'(0))=(u,s).
$$
Then
$$
\frac{d}{d\tau}\begin{pmatrix}
1 & 0\\
v(\tau) & e^{t(\tau)\operatorname{J}(\aleph)}
\end{pmatrix}\biggr\rvert_{\tau=0}=\begin{pmatrix}
0 & 0\\
u & s\operatorname{J}(\aleph)
\end{pmatrix}\in\aA(\aleph),
$$
which proves that $\aA(\aleph)$ is the Lie algebra of $G$. Finally, by construction $G$ is diffeomorphic to $\mathbb{R}^d\times(\mathbb{R}/T_\aleph)$, which is simply connected iff $T_\aleph$ is trivial. $\Box$
\end{proof}

If the Lie group in Proposition \ref{NaiveMatRepProp} is simply connected then it is a faithful matrix representation for the simply connected almost Abelian Lie group with Lie algebra $\aA(\aleph)$. But there is a simple modification that yields a faithful matrix representation for every simply connected almost Abelian Lie group.

\begin{proposition}\label{SCMatRepProp} For a finite multiplicity function $\aleph$ let
$$
G_\mathrm{I}\doteq\left\{\begin{pmatrix}
1 & 0 & 0\\
v & e^{t\operatorname{J}(\aleph)} & 0\\
0 & 0 & e^t
\end{pmatrix}\,\vline\quad(v,t)\in\mathbb{R}^d\oplus\mathbb{R}\right\},\quad\quad G_\mathrm{II}\doteq\left\{\begin{pmatrix}
1 & 0 & 0\\
v & e^{t\operatorname{J}(\aleph)} & 0\\
t & 0 & 1
\end{pmatrix}\,\vline\quad(v,t)\in\mathbb{R}^d\oplus\mathbb{R}\right\}.
$$
Then both $G_\mathrm{I}$ and $G_\mathrm{II}$ are simply connected Lie groups with Lie algebras isomorphic to $\aA(\aleph)$.
\end{proposition}
\begin{proof} That $G_\mathrm{I}$ and $G_\mathrm{II}$ are Lie groups is clear from the definitions. The map
$$
\mathbb{R}^d\oplus\mathbb{R}\ni(v,t)\mapsto\begin{pmatrix}
1 & 0 & 0\\
v & e^{t\operatorname{J}(\aleph)} & 0\\
t & 0 & 1
\end{pmatrix}\in G_\mathrm{II}
$$
is a diffeomorphism, which proves that $G_\mathrm{II}$ is simply connected. For $\forall (u,s)\in\mathbb{R}^d\oplus\mathbb{R}$ let
$$
(-1,1)\ni\tau\mapsto\begin{pmatrix}
1 & 0 & 0\\
v(\tau) & e^{t(\tau)\operatorname{J}(\aleph)} & 0\\
t(\tau) & 0 & 1
\end{pmatrix}\in G_\mathrm{II}
$$
be a smooth curve with
$$
(v(0),t(0))=(0,0),\quad (v'(0),t'(0))=(u,s).
$$
Then
$$
\frac{d}{d\tau}\begin{pmatrix}
1 & 0 & 0\\
v(\tau) & e^{t(\tau)\operatorname{J}(\aleph)} & 0\\
t(\tau) & 0 & 1
\end{pmatrix}\biggr\rvert_{\tau=0}=\begin{pmatrix}
0 & 0 & 0\\
u & s\operatorname{J}(\aleph) & 0\\
s & 0 & 0
\end{pmatrix},
$$
which through faithful representation (\ref{LieAlgFaithRep3}) shows the isomorphism between $\aA(\aleph)$ and the Lie algebra of $G_\mathrm{II}$. Finally, the map
$$
G_\mathrm{I}\ni\begin{pmatrix}
1 & 0 & 0\\
v & e^{t\operatorname{J}(\aleph)} & 0\\
0 & 0 & e^t
\end{pmatrix}\quad\mapsto\quad\begin{pmatrix}
1 & 0 & 0\\
v & e^{t\operatorname{J}(\aleph)} & 0\\
t & 0 & 1
\end{pmatrix}\in G_\mathrm{II}
$$
is easily checked to be a Lie group isomorphism, which shows that the above statements hold for the Lie group $G_\mathrm{I}$ as well. $\Box$
\end{proof}
Thus, a simply connected almost Abelian Lie group is a semidirect product $G=\mathbb{R}^d\rtimes\mathbb{R}$, which is consistent with the Lie algebra being a semidirect product $\aA(\aleph)=\mathbb{R}^d\rtimes\mathbb{R}$. In order to notationally distinguish between a Lie algebra element in $\mathbb{R}^d\rtimes\mathbb{R}$ and a Lie group element in $\mathbb{R}^d\rtimes\mathbb{R}$ we will use $(v,t)\in\mathbb{R}^d\rtimes\mathbb{R}$ for the former and $[v,t]\in\mathbb{R}^d\rtimes\mathbb{R}$ for the latter, respectively.

%\end{document}

\quad

%\documentclass{article}

%\input{include.tex}

%\begin{document}

\section*{Properties of the exponential map}\label{ExpMap}
%\subsubsection*{M. Almora, Z. Avetisyan, K. Berlow}

We continue working with a real finite-dimensional almost Abelian Lie algebra $\aA(\aleph)$ and the associated simply connected almost Abelian Lie group $G$. Below we will establish technical facts that will together yield the necessary and sufficient condition for the exponential map on $G$ to be a diffeomorphism. It is well known that a solvable real simply connected Lie group fails to be exponential if and only if it contains a copy of $\widetilde{\mathrm{E}_+(2)}$ - the universal cover of the identity component of the Euclidean motion group in $\mathbb{R}^2$. Here we will reestablish this fact in a much more explicit way for the particular case of an almost Abelian simply connected Lie group, and find an equivalent condition in terms of the spectrum of the adjoint representation $\ad_{e_0}=\operatorname{J}(\aleph)$.

In order to describe the exponential map of the simply connected almost Abelian group $G$ in terms of matrix exponentials we need to carefully choose faithful matrix representations for the group $G$ and the Lie algebra $\aA(\aleph)$. Namely, we will choose the algebra representation (\ref{LieAlgFaithRep2}) for the group representation $G_\mathrm{I}$, and the algebra representation (\ref{LieAlgFaithRep3}) for the group representation $G_\mathrm{II}$, respectively (in terms of Proposition \ref{SCMatRepProp}).

\begin{lemma}\label{ExpMapLemma} The exponential map of the simply connected almost Abelian group $G$ corresponding to the almost Abelian Lie algebra $\aA(\aleph)$ can be given by
$$
\exp\begin{pmatrix}
0 & 0 & 0\\
v & t\operatorname{J}(\aleph) & 0\\
0 & 0 & t
\end{pmatrix}\quad=\quad\begin{pmatrix}
1 & 0 & 0\\
\frac{e^{t\operatorname{J}(\aleph)}-\id}{t\operatorname{J}(\aleph)}v & e^{t\operatorname{J}(\aleph)} & 0\\
0 & 0 & e^t
\end{pmatrix}\in G_\mathrm{I}
$$
or
$$
\exp\begin{pmatrix}
0 & 0 & 0\\
v & t\operatorname{J}(\aleph) & 0\\
t & 0 & 0
\end{pmatrix}\quad=\quad\begin{pmatrix}
1 & 0 & 0\\
\frac{e^{t\operatorname{J}(\aleph)}-\id}{t\operatorname{J}(\aleph)}v & e^{t\operatorname{J}(\aleph)} & 0\\
t & 0 & 1
\end{pmatrix}\in G_\mathrm{II}.
$$
\end{lemma}
\begin{proof} It is sufficient to perform matrix exponentiation, which we will do only for $G_\mathrm{II}$, since for $G_\mathrm{I}$ it is very similar. We first observe that
\begin{equation}
\begin{pmatrix}
0 & 0 & 0\\
v & t\operatorname{J}(\aleph) & 0\\
t & 0 & 0
\end{pmatrix}^n=\begin{pmatrix}
0 & 0 & 0\\
[t\operatorname{J}(\aleph)]^{n-1}v & [t\operatorname{J}(\aleph)]^n & 0\\
0 & 0 & 0
\end{pmatrix},\quad\forall n>1,\label{X^n}
\end{equation}
by mathematical induction on $n$. Then by series expansion of the exponential, we see that
\begin{eqnarray}
\exp\begin{pmatrix}
0 & 0 & 0\\
v & t\operatorname{J}(\aleph) & 0\\
t & 0 & 0
\end{pmatrix}\quad=\sum_{n=0}^\infty\frac1{n!}\begin{pmatrix}
0 & 0 & 0\\
v & t\operatorname{J}(\aleph) & 0\\
t & 0 & 0
\end{pmatrix}^n\nonumber\\
=\id+\begin{pmatrix}
0 & 0 & 0\\
v & t\operatorname{J}(\aleph) & 0\\
t & 0 & 0
\end{pmatrix}+\sum_{n=2}^\infty\frac1{n!}\begin{pmatrix}
0 & 0 & 0\\
[t\operatorname{J}(\aleph)]^{n-1}v & [t\operatorname{J}(\aleph)]^n & 0\\
0 & 0 & 0
\end{pmatrix}=\begin{pmatrix}
1 & 0 & 0\\
\frac{e^{t\operatorname{J}(\aleph)}-\id}{t\operatorname{J}(\aleph)}v & e^{t\operatorname{J}(\aleph)} & 0\\
t & 0 & 1
\end{pmatrix}
\end{eqnarray}
gives us the desired result. $\Box$
\end{proof}
Note that Lemma \ref{ExpMapLemma} can now be written as
$$
\exp(v,t)=\left[\frac{e^{t\operatorname{J}(\aleph)}-\id}{t\operatorname{J}(\aleph)}v,t\right],\quad\forall(v,t)\in\aA(\aleph).
$$

\begin{remark}\label{ExpZRemark} It follows that on the Abelian Lie subalgebra $\ker\operatorname{J}(\aleph)\oplus\mathbb{R}$ the exponential map is $$
\exp(v,t)=[v,t],\quad\forall (v,t)\in\ker\operatorname{J}(\aleph)\oplus\mathbb{R}.
$$
\end{remark}

For the next fact note that the Lie algebra of $\mathrm{E}_+(2)$ is $\aA(1\times\imath^1)$.

\begin{lemma}\label{E2=>nexpLemma} If the almost Abelian Lie algebra $\aA(\aleph)$ contains a subalgebra $\mathbf{L}\subseteq\aA(\aleph)$ isomorphic to $\aA(1\times\imath^1)$ then the corresponding simply connected almost Abelian Lie group $G$ is not exponential.
\end{lemma}
\begin{proof} Let $\varphi:\aA(1\times\imath^1)\to\mathbf{L}$ be a Lie algebra isomorphism, and let $H\subset G$ be the connected Lie subgroup with associated Lie algebra $\mathbf{L}$ as given by Theorem 5.20 in \cite{Hall2015}. Let $\exp_1:\aA(1\times\imath^1)\to\widetilde{\mathrm{E}_+(2)}$ be the exponential map on $\widetilde{\mathrm{E}_+(2)}$, and $\exp_2:\mathbf{L}\to H$ be the exponential map on $H$. Assume towards a contradiction that $\exp_2$ is injective. Since $\exp_2$ is injective, $H$ is simply connected. Thus since $H$ and $\widetilde{E_+}(2)$ are both simply connected, by Theorem 5.6 in \cite{Hall2015}, there is a Lie group isomorphism $\Phi:\widetilde{\mathrm{E}_+(2)}\to H$ such that $\Phi(\exp_1(X)) = \exp_2(\varphi(X))$ for all $X\in\aA(1\times\imath^1)$. Since we know that $\exp_1:\aA(1\times\imath^1)\to\widetilde{\mathrm{E}_+(2)}$ is not injective, let $X,Y\in\aA(1\times\imath^1)$ such that $X\neq Y$ and $\exp_1(X)= \exp_1(Y)$. Then $\Phi(\exp_1(X)) = \Phi(\exp_1(Y))$ and therefore $\exp_2(\varphi(X)) = \exp_2(\varphi(Y))$. Since $\exp_2$ is injective, this implies that $\varphi(X) = \varphi(Y)$ in spite of $X \neq Y$, thus contradicting the assumption of $\varphi$ being an isomorphism. This contradiction proves that $\exp_1$ and therefore also $\exp:\aA(\aleph)\to G$ cannot be injective, and $G$ is not exponential. $\Box$
\end{proof}

\begin{lemma}\label{nexp<=>iRLemma} The simply connected Lie group $G$ with almost Abelian Lie algebra $\aA(\aleph)$ fails to be exponential if and only if $\supp\aleph$ contains a polynomial $p$ with non-zero imaginary root $x_p$.
\end{lemma}
\begin{proof} Let us perform some preliminary computations first. From \cite{Avetisyan2018}
\begin{equation}
\operatorname{J}(\aleph)=\bigoplus_{p\in\supp\aleph}\bigoplus_{n=1}^\infty\bigoplus_{\aleph(p,n)}\operatorname{J}(p,n).\label{JBlock}
\end{equation}
Thus, the exponential of the Jordan canonical form above can similarly be decomposed as
\begin{equation}
e^{t\operatorname{J}(\aleph)}=\bigoplus_{p\in\supp\aleph}\bigoplus_{n=1}^\infty\bigoplus_{\aleph(p,n)}e^{t\operatorname{J}(p,n)}.\label{ExpJBlock}
\end{equation}
Now $G$ is not exponential iff
\begin{equation}
\exists(v_1,t_1),(v_2,t_2)\in\aA(\aleph)\quad\mbox{\normalfont s.t.}\quad(v_1,t_1)\neq(v_2,t_2),\quad \exp(v_1,t_1)=\exp(v_2,t_2).\label{existsv1t1v2t2}
\end{equation}
From Lemma \ref{ExpMapLemma} we see that $\exp(v_1,t_1)=\exp(v_2,t_2)$ if and only if $t_1=t_2\doteq t$ and
\begin{equation}
\frac{e^{t\operatorname{J}(\aleph)}-\id}{t\operatorname{J}(\aleph)}(v_1-v_2)=0.\label{exp1=exp2}
\end{equation}
Therefore (\ref{existsv1t1v2t2}) is equivalent to
$$
\exists\,t\in\mathbb{R}\quad\mbox{\normalfont s.t.}\quad\det\left[\frac{e^{t\operatorname{J}(\aleph)}-\id}{t\operatorname{J}(\aleph)}\right]=0.
$$
From (\ref{JBlock}) we find that
$$
\det\left[\frac{e^{t\operatorname{J}(\aleph)}-\id}{t\operatorname{J}(\aleph)}\right]=\prod_{p\in\supp\aleph}\prod_{n=1}^\infty\det\left[\frac{e^{t\operatorname{J}(p,n)}-\id}{t\operatorname{J}(p,n)}\right]^{\aleph(p,n)}.
$$
Thus, (\ref{existsv1t1v2t2}) is equivalent to
$$
\exists\,t\in\mathbb{R},\quad\exists p\in\supp\aleph,\quad\exists n\in\mathbb{N}\quad\mbox{\normalfont s.t.}\quad\aleph(p,n)>0\quad\wedge\quad\det\left[\frac{e^{t\operatorname{J}(p,n)}-\id}{t\operatorname{J}(p,n)}\right]=0.
$$
If $x_p\in\mathbb{R}$ then
$$
\det\left[\frac{e^{t\operatorname{J}(p,n)}-\id}{t\operatorname{J}(p,n)}\right]=\left(\frac{e^{tx_p}-\id}{tx_p}\right)^n>0,
$$
whereas if $x_p=a+\imath b\not\in\mathbb{R}$ then
$$
\det\left[\frac{e^{t\operatorname{J}(p,n)}-\id}{t\operatorname{J}(p,n)}\right]=\left(\frac{(e^{ta}\cos(tb)-1)^2+(e^{ta}\sin(tb))^2}{t^2(a^2+b^2)}\right)^n.
$$
Therefore
$$
\det\left[\frac{e^{t\operatorname{J}(\aleph)}-\id}{t\operatorname{J}(\aleph)}\right]=0\qquad\Leftrightarrow\qquad\exists\,p\in\supp\aleph\quad\mbox{s.t.}\quad\frac{tx_p}{2\pi}\in\imath\mathbb{Z}.
$$
Thus (\ref{existsv1t1v2t2}) is equivalent to
$$
\exists\,p\in\supp\aleph\quad\mbox{s.t.}\quad0\neq x_p\in\imath\mathbb{R},
$$
exactly as in the statement of the lemma. $\Box$
\end{proof}

\begin{lemma}\label{iR=>E2Lemma} If $\supp\aleph$ contains a polynomial $p$ with non-zero imaginary root $x_p$ then there exists a Lie subalgebra $\mathbf{L}\subset\aA(\aleph)$ which is isomorphic to $\aA(1\times\imath^1)$.
\end{lemma}
\begin{proof} Suppose that $\exists\,p\in\supp\aleph$ such that $x_p=\imath b$ with $0\neq b\in\mathbb{R}$ and $\aleph(p,n)>0$ for some $n\in\mathbb{N}$. Fix an $\alpha\in\aleph(p,n)$ and let $\{\xi_\alpha^i(p,n)\}_{i=1}^n$ be the standard basis in the Jordan block $(p,n,\alpha)$ as in \cite{Avetisyan2018}. Let $\mathbf{W}=\mathbb{C}\{\xi_\alpha^1(p,n)\}$ as an $\mathbb{R}$-vector space. By Corollary 3 in \cite{Avetisyan2018}, $\mathbf{W}$ is an $\ad_{e_0}$-invariant subspace, and the restriction $\ad_{e_0}|_\mathbf{W}=x_p=\imath b$, which is $\mathbb{R}$-projectively similar to $\imath$ on $\mathbb{C}$. Thus the Lie subalgebra $\mathbf{L}\rtimes\mathbb{R}e_0\subset\aA(\aleph)$ is nothing else but $\aA(1\times\imath b^1)$, which by Proposition 11 in \cite{Avetisyan2016} is isomorphic to $\aA(1\times\imath^1)$. $\Box$
\end{proof}

Finally we are ready to formulate the main result of this section.

\begin{proposition} A simply connected almost Abelian Lie group with Lie algebra $\aA(\aleph)$ fails to be exponential if and only if $\supp\aleph$ contains a non-zero purely imaginary number, which is equivalent to the existence of a Lie subalgebra isomorphic to $\aA(1\times\imath^1)$.
\end{proposition}
\begin{proof} Follows directly by combining Lemma \ref{E2=>nexpLemma}, Lemma \ref{nexp<=>iRLemma} and Lemma \ref{iR=>E2Lemma}. $\Box$
\end{proof}

In our further studies we will need a precise description of the set $T_\aleph$ defined in (\ref{TalephDef}).

\begin{lemma}\label{TalephLemma} For a given finite real multiplicity function $\aleph$, we have $T_\aleph\neq\{0\}$ if and only if $\aleph(p,n)=0$ for all $p\in\supp\aleph$ and $n>1$ and $\mathcal{X}_\aleph\neq\emptyset$, in which case
$$
T_\aleph=\frac{2\pi}{\omega_0}\mathbb{Z},\quad \omega_0\in\mathcal{X}_\aleph,\quad|\omega_0|=\max\left\{|\omega|\,\vline\quad\omega\in\mathcal{X}_\aleph\right\}.
$$
\end{lemma}
\begin{proof} We recall from \cite{Avetisyan2018} that $\operatorname{J}(p,n)=x_p\id+\operatorname{N}_n$ understood over the field $\mathbb{R}(x_p)\subset\End(\mathbb{R}^{\deg p})$. Continuing from (\ref{ExpJBlock}) we find that
$$
e^{t\operatorname{J}(\aleph)}=\bigoplus_{p\in\supp\aleph}\bigoplus_{n=1}^\infty\bigoplus_{\aleph(p,n)}e^{tx_p}\left(\id+t\operatorname{N}_n+\frac12t^2\operatorname{N}_n^2+\ldots\right),
$$
whence
$$
e^{t\operatorname{J}(\aleph)}=\id\qquad\Leftrightarrow\qquad\left[t=0\quad\mbox{or}\quad\forall p\in\supp\aleph,\,x_p=\imath b_p\in\imath\mathbb{R},\,e^{\imath t b_p}=1,\,\aleph(p,n)=0,\,\forall n>1\right],
$$
which means that
$$
T_\aleph\neq\{0\}\qquad\Leftrightarrow\qquad\exists t\neq0\quad\mbox{s.t.}\quad\forall p\in\supp\aleph,\,x_p=\imath b_p\in\imath\mathbb{R},\,e^{\imath t b_p}=1,\,\aleph(p,n)=0,\,\forall n>1.
$$
Let us show that provided $\aleph(p,n)=0,\,\forall p\in\supp\aleph,\,\forall n>1$, we have
$$
\exists t\neq0\quad\mbox{s.t.}\quad\forall p\in\supp\aleph,\,x_p=\imath b_p\in\imath\mathbb{R},\,e^{\imath t b_p}=1\qquad\Leftrightarrow\qquad\mathcal{X}_\aleph\neq\emptyset.
$$
Indeed, if $t\neq0$ then the condition $e^{\imath t b_p}=1$ can be written as
$$
x_p\in\imath\frac{2\pi}t\mathbb{Z},\quad\forall p\in\supp\aleph,
$$
which implies that
$$
\frac{2\pi}{t}\in\mathcal{X}_\aleph.
$$
Conversely, let $0\neq\omega\in\mathcal{X}_\aleph$. The possibility $\omega=0$ is excluded, since in that case $\supp\aleph=\{0\}$, which together with $n=1$ would imply that $\operatorname{J}(\aleph)=0$, i.e., that the Lie algebra is Abelian. Thus $\omega\neq0$, and setting $t=2\pi/\omega$ we check that $e^{t\operatorname{J}(\aleph)}=\id$, i.e., $t\in T_\aleph$.

Finally, $T_\aleph\subset\mathbb{R}$ is the kernel of the homomorphism $\mathbb{R}\ni t\mapsto e^{t\operatorname{J}(\aleph)}\in\Aut(\mathbb{R}^d)$, and is therefore a discrete subgroup of the form
$$
T_\aleph=t_0\mathbb{Z},\quad |t_0|=\min\left\{|t|\,{\big|}\quad0\neq t\in T_\aleph\right\}.
$$
Since $0\neq\omega\in\mathcal{X}_\aleph$ is equivalent to $2\pi/\omega\in T_\aleph$, we have that
$$
t_0=\frac{2\pi}{\omega_0},\quad|\omega_0|=\max\left\{|\omega|\,\vline\quad\omega\in\mathcal{X}_\aleph\right\},
$$
which completes the proof. $\Box$\
\end{proof}

%\end{document}

\quad

%\documentclass{article}

%\input{include.tex}

%\begin{document}

\section*{Discrete normal subgroups and quotients of simply connected almost Abelian groups}\label{DNSubGr}
%\subsubsection*{Z. Avetisyan, I. Martin, Z. Zhao}

In this section we will describe explicitly the discrete normal subgroups $N$ of a simply connected almost Abelian Lie group $G$. Then we will derive a necessary and sufficient condition for two quotient groups $G/N$ to be isomorphic.

We start by describing the centre of a simply connected almost Abelian Lie group. Recall from \cite{Avetisyan2016} and \cite{Avetisyan2018} that the centre of an almost Abelian Lie algebra $\aA(\aleph)$ is
$$
\mathrm{Z}(\aA(\aleph))=\ker\operatorname{J}(\aleph),
$$
and denote
$$
T_\aleph\doteq\left\{t\in\mathbb{R}\,{\big|}\quad e^{t\operatorname{J}(\aleph)}=\id\right\}\subset\mathbb{R}.
$$

\begin{proposition}\label{ZGProp} The centre of the simply connected almost Abelian Lie group $G$ with Lie algebra $\aA(\aleph)$ is given by
$$
\mathrm{Z}(G)=\exp\bigl[\mathrm{Z}(\aA(\aleph))\bigr]\times T_\aleph=\exp\bigl[\mathrm{Z}(\aA(\aleph))\times T_\aleph\bigr]
$$
$$
=\left\{[u,s]\in\mathbb{R}^d\rtimes\mathbb{R}\,{\big|}\quad u\in\ker\operatorname{J}(\aleph),\quad e^{s\operatorname{J}(\aleph)}=\id\right\}.
$$
The preimage of the identity component of the centre through the exponential map is
$$
\exp^{-1}\bigl[\mathrm{Z}(G)_0\bigr]=\mathrm{Z}(\aA(\aleph)).
$$
\end{proposition}
\begin{proof} Let us use the faithful matrix representation
$$
G=\mathbb{R}^n\rtimes\mathbb{R}\ni[v,t]\quad=\quad\begin{pmatrix}
1 & 0 & 0\\
v & e^{t\operatorname{J}(\aleph)} & 0\\
t & 0 & 1
\end{pmatrix}
$$
provided by Proposition \ref{SCMatRepProp}. Suppose that $[u,s]\in\mathrm{Z}(G)$. Then the following must be satisfied,
$$
[v,t][u,s]=\begin{pmatrix}
1 & 0 & 0\\
v & e^{t\operatorname{J}(\aleph)} & 0\\
t & 0 & 1
\end{pmatrix}\begin{pmatrix}
1 & 0 & 0\\
u & e^{s\operatorname{J}(\aleph)} & 0\\
s & 0 & 1
\end{pmatrix}=\begin{pmatrix}
1 & 0 & 0\\
v+e^{t\operatorname{J}(\aleph)}u & e^{(t+s)\operatorname{J}(\aleph)} & 0\\
t+s & 0 & 1
\end{pmatrix}
$$
$$
=\begin{pmatrix}
1 & 0 & 0\\
u+e^{s\operatorname{J}(\aleph)}v & e^{(t+s)\operatorname{J}(\aleph)} & 0\\
t+s & 0 & 1
\end{pmatrix}=\begin{pmatrix}
1 & 0 & 0\\
u & e^{s\operatorname{J}(\aleph)} & 0\\
s & 0 & 1
\end{pmatrix}\begin{pmatrix}
1 & 0 & 0\\
v & e^{t\operatorname{J}(\aleph)} & 0\\
t & 0 & 1
\end{pmatrix}=[u,s][v,t],\quad\forall(v,t)\in G.
$$
This is equivalent to $v+e^{t\operatorname{J}(\aleph)}u=u+e^{s\operatorname{J}(\aleph)}v$ or
$$
\left(e^{t\operatorname{J}(\aleph)}-\id\right)u=\left(e^{s\operatorname{J}(\aleph)}-\id\right)v,\quad\forall[v,t]\in G.
$$
Setting $v=0$ we have that $(e^{t\operatorname{J}(\aleph)}-\id)u=0$ which forces $\operatorname{J}(\aleph)u=0$ or $u\in\ker\operatorname{J}(\aleph)$, as desired. But if $u$ is such then $(e^{s\operatorname{J}(\aleph)}-\id)v=0$ for all $v$, which means that $e^{s\operatorname{J}(\aleph)}=\id$. The first statement of the proposition now follows from Remark \ref{ExpZRemark}. If $\exp(v,t)=[u,s]\in\mathrm{Z}(G)_0$ then $t=s=0$ and $v=u$, as desired. $\Box$
\end{proof}

Now let us proceed to the discrete normal subgroups $N\subset G$ of a simply connected almost Abelian Lie group.

\begin{proposition}\label{NProp} Every discrete normal subgroup $N\subset G$ of a simply connected almost Abelian Lie group $G$ with Lie algebra $\aA(\aleph)$ is a free group of rank $k\le\dim\ker\operatorname{J}(\aleph)+1$ generated by $\mathbb{R}$-linearly independent elements
$$
[v_1,t_1],\ldots,[v_k,t_k]\in\mathrm{Z}(G)\subset G=\mathbb{R}^d\rtimes\mathbb{R}.
$$
\end{proposition}
\begin{proof} It is well known that every discrete normal subgroup of a connected Lie group is in fact central (e.g., \cite{Hall2015}). Thus it suffices to find discrete subgroups of $\mathrm{Z}(G)$. Notice that for every $[v,t],[u,s]\in\mathrm{Z}(G)$, $[v,t][u,s]=[v+u,t+s]$, so that the restriction of the obvious homeomorphism $f:G\to\mathbb{R}^{d+1}$ to $\mathrm{Z}(G)$ is also an injective Lie group homomorphism
$$
f|_{\mathrm{Z}(G)}:\mathrm{Z}(G)\to\mathbb{R}^{d+1}.
$$
Therefore every discrete subgroup $N\subset\mathrm{Z}(G)$ is mapped to a discrete subgroup $f(N)\subset\mathbb{R}^{d+1}$. As a discrete subgroup of $\mathbb{R}^{d+1}$, $f(N)$ is a free Abelian group generated by $\mathbb{R}$-linearly independent elements $\nu_1,\ldots,\nu_k\in\mathbb{R}^{d+1}$, and their span satisfies
$$
\mathbb{R}\{\nu_i\}_{i=1}^k\subset\mathbb{R}\left\{f(\mathrm{Z}(G))\right\},
$$
which implies that
$$
k\le\dim\mathbb{R}\left\{f(\mathrm{Z}(G))\right\}\le\dim\ker\operatorname{J}(\aleph)+1.
$$
Setting $[v_i,t_i]\doteq f^{-1}(\nu_i)$ for $i=1,\ldots,k$ completes the proof. $\Box$
\end{proof}

Now that we have a description of discrete normal subgroups $N\subseteq G$ of a simply connected almost Abelian Lie group, and since every connected almost Abelian Lie group can be written as a quotient $G/N$ for a corresponding $N$, we have effectively covered all connected almost Abelian Lie groups. Next we want to know for which distinct discrete normal subgroups $N,M\subset G$ the quotient groups $G/N$ and $G/M$ are isomorphic. Below is a pretty quantitative answer to this question. Denote by $\operatorname{q}_N:G\to G/N$ and $\operatorname{q}_M:G\to G/M$ the canonical quotient homomorphisms, and by $\Hom^*(G/N,G/M)$ the set of all Lie group isomorphisms $G/N\to G/M$.

\begin{proposition}\label{G_N=G_M} Let $G$ be a simply connected Lie group and $N,M\subset G$ two discrete normal subgroups. Then
$$
\Hom^*(G/N,G/M)=\left\{\Phi_\NM=\operatorname{q}_M\circ\Phi\circ\operatorname{q}_N^{-1}\,{\big|}\quad\Phi\in\Aut(G),\quad\Phi(N)=M\right\}.
$$
\end{proposition}
\begin{proof} We first prove that
\begin{equation}
\bigl[\,\exists\Phi_\NM\in\Hom^*(G/N,G/M)\quad\mbox{s.t.}\quad\Phi_\NM\circ\operatorname{q}_N=\operatorname{q}_M\circ\Phi\,\bigr]\quad\Leftrightarrow\quad\Phi(N)=M,\quad\quad\forall\Phi\in\Aut(G).\label{PhiN=M}
\end{equation}
Let $\Phi_\NM$ as above be given. Then
$$
\operatorname{q}_M\circ\Phi(n)=\Phi_\NM\circ\operatorname{q}_N(n)=\id,\quad\forall n\in N,
$$
whence $\Phi(n)\in M$, $\forall n\in N$, and thus $\Phi(N)\subset M$. But also
$$
\operatorname{q}_N(\Phi^{-1}(m))=\Phi_\NM^{-1}\circ\operatorname{q}_M\circ\Phi\left(\Phi^{-1}(m)\right)=\Phi_\NM^{-1}\circ\operatorname{q}_M(m)=\id,\quad\forall m\in M,
$$
so that $\Phi^{-1}(m)\in N$, $\forall m\in M$, and thus $\Phi^{-1}(M)\subset N$. We conclude that $\Phi(N)=M$. Conversely, assume that $\Phi(N)=M$. Then
$$
\operatorname{q}_M\circ\Phi(\operatorname{q}_N^{-1}(\id))=\operatorname{q}_M(\Phi(N))=\operatorname{q}_M(M)=\id,
$$
so that $\Phi_\NM\doteq\operatorname{q}_M\circ\Phi\circ\operatorname{q}_N^{-1}:G/N\to G/M$ is well defined. This completes the proof of (\ref{PhiN=M}). It remains to show that every isomorphism $\Psi\in\Hom^*(G/N,G/M)$ arises as $\Psi=\Phi_\NM$ for a unique $\Phi\in\Aut(G)$. To see this let $d\Psi$ be the corresponding Lie algebra automorphism (say, Theorem 3.28 in \cite{Hall2015}).  Then since $G$ is simply connected and has the same Lie algebra as $G/N$ and $G/M$, there exists a unique $\Phi\in\Aut(G)$ such that $d\Phi=d\Psi$. Consider the following two Lie group homomorphisms,
$$
\Psi\circ\operatorname{q}_N:G\to G/N,\quad\operatorname{q}_M\circ\Phi:G\to G/M.
$$
By Proposition 3.30 in \cite{Hall2015},
$$
d\left(\Psi\circ\operatorname{q}_N\right)=d\Psi\circ d\operatorname{q}_N=d\Psi=d\Phi=d\operatorname{q}_M\circ d\Phi=d\left(\operatorname{q}_M\circ\Phi\right).
$$
But then by uniqueness in Theorem 5.6 of \cite{Hall2015} it follows that $\Psi\circ\operatorname{q}_N=\operatorname{q}_M\circ\Phi$, and that $\Phi$ is unique with this property. The assertion is proven. $\Box$
\end{proof}
\noindent In particular, two quotient groups $G/N$ and $G/M$ are isomorphic if and only if $\Hom^*(G/N,G/M)\neq\emptyset$.

%\end{document}

\quad

%\documentclass{article}

%\input{include.tex}

%\begin{document}

\section*{Automorphisms of almost Abelian Lie groups}\label{Aut}
%\subsubsection*{Z. Avetisyan, K. Berlow, I. Martin}

In this section we will find an explicit description of the automorphism group $\Aut(G)$ of a connected almost Abelian Lie group $G$, with each automorphism given as a diffeomorphism in global group coordinates. For this purpose we will first combine Proposition 7, Proposition 8, Proposition 9 and Proposition 10 from \cite{Avetisyan2016} into a single convenient description of automorphisms of an almost Abelian Lie algebra.

\begin{proposition} The automorphism group $\Aut(\aA(\aleph))\subset\End(\mathbb{R}^d\rtimes\mathbb{R})$ of a real almost Abelian Lie algebra $\aA(\aleph)=\mathbb{R}^d\rtimes\mathbb{R}$ takes the form
\begin{equation}
\Aut(\aA(\aleph))=\left\{\begin{pmatrix}
x\\
y\\
t\\
w
\end{pmatrix}\longmapsto\begin{pmatrix}
\alpha\Delta_{22}-\beta_2\gamma_2 & \Delta_{12} & \gamma_1 & \phi_{01}\\
0 & \Delta_{22} & \gamma_2 & 0\\
0 & \beta_2 & \alpha & 0\\
0 & \eta & \rho & \phi_{11}
\end{pmatrix}\begin{pmatrix}
x\\
y\\
t\\
w
\end{pmatrix}\,\vline\quad\begin{matrix}
\alpha,\beta_2,\gamma_1,\gamma_2,\Delta_{12},\Delta_{22}\in\mathbb{R},\\
\alpha\Delta_{22}-\beta_2\gamma_2\neq0,\, \eta,\rho\in\mathbb{R}^{d-2},\\
\phi_{01}\in\Hom(\mathbb{R}^{d-2},\mathbb{R}),\\
\phi_{11}\in\Aut(\mathbb{R}^{d-2})
\end{matrix}\right\}\label{AutH+R}
\end{equation}
if $\aA(\aleph)=\mathbf{H}\oplus\mathbb{R}^{d-2}$ is a central extension of the Heisenberg algebra and
\begin{equation}
\Aut(\aA(\aleph))=\left\{\begin{pmatrix}
\Delta & \gamma\\
0 & \alpha
\end{pmatrix}\,\vline\quad\alpha\in\Dil(\aleph),\quad\gamma\in\mathbb{R}^d,\quad\Delta\in\Aut(\mathbb{R}^d),\quad\Delta\operatorname{J}(\aleph)=\alpha\operatorname{J}(\aleph)\Delta\right\}\label{AutL}
\end{equation}
otherwise.
\end{proposition}

\begin{remark} If we apply formula (\ref{AutL}) to the Lie algebra $\mathbf{H}\oplus\mathbb{R}^{d-2}$ then we will obtain only the subgroup consisting of those automorphisms corresponding to $\beta_2=0$ in formula (\ref{AutH+R}).
\end{remark}

We begin with the case of a simply connected $G$, where there is a bijective correspondence between Lie algebra automorphisms and Lie group automorphisms. On several occasions we will make use of the following elementary fact.

\begin{remark}\label{FABCRemark} If $A$, $B$ and $C$ are square matrices such that $AB=BC$ then for every entire holomorphic function $F\in\Hol(\mathbb{C})$ one has $F(A)B=BF(C)$.
\end{remark}
This can be easily checked term by term in the Taylor expansion.

Let $\mathcal{H}=\exp(\mathbf{H})$ stand for the Heisenberg group.

\begin{proposition}\label{AutGProp} If $G$ is a simply connected almost Abelian Lie group with Lie algebra $\aA(\aleph)$ then
\begin{eqnarray}
\Aut(G)=\left\{\begin{bmatrix}
x\\
y\\
t\\
w
\end{bmatrix}\overset{\Phi}{\longmapsto}\begin{bmatrix}
\left[\alpha\Delta_{22}-\beta_2\gamma_2\right]x+\Delta_{12}y+\gamma_1t+\beta_2\gamma_2ty+\frac12\alpha\gamma_2t^2+\frac12\Delta_{22}\beta_2y^2+\phi_{01}(w)\\
\Delta_{22}y+\gamma_2t\\
\beta_2y+\alpha t\\
\eta y+\rho t+\phi_{11}(w)
\end{bmatrix}\vline\right.\nonumber\\
\left.d\Phi{\big|}_{(0,0)}=\begin{pmatrix}
\alpha\Delta_{22}-\beta_2\gamma_2 & \Delta_{12} & \gamma_1 & \phi_{01}\\
0 & \Delta_{22} & \gamma_2 & 0\\
0 & \beta_2 & \alpha & 0\\
0 & \eta & \rho & \phi_{11}
\end{pmatrix}\in\Aut(\aA(\aleph))\right\}\label{AutHxR}
\end{eqnarray}
if $G=\mathcal{H}\times\mathbb{R}^{d-2}$ is a central extension of the Heisenberg group and
\begin{equation}
\Aut(G)=\left\{[v,t]\overset{\Phi}{\longmapsto}\left[\frac{e^{\alpha t\operatorname{J}(\aleph)}-\id}{\alpha\operatorname{J}(\aleph)}\gamma+\Delta v,\alpha t\right]\,\vline\quad d\Phi{\big|}_{(0,0)}=\begin{pmatrix}
\Delta & \gamma\\
0 & \alpha
\end{pmatrix}\in\Aut(\aA(\aleph))\right\}\label{AutG}
\end{equation}
otherwise.
\end{proposition}
\begin{proof} Central extensions of the Heisenberg group are exponential, and we can use the bijectivity of the exponential map to switch from Lie algebra automorphisms to Lie group automorphisms. Namely, if
$$
d\Phi\begin{pmatrix}
x\\
y\\
t\\
w
\end{pmatrix}=\begin{pmatrix}
\left[\alpha\Delta_{22}-\beta_2\gamma_2\right]x+\Delta_{12}y+\gamma_1t+\phi_{01}(w)\\
\Delta_{22}y+\gamma_2t\\
\beta_2y+\alpha t\\
\eta y+\rho t+\phi_{11}(w)
\end{pmatrix}
$$
then
$$
\Phi\left(\exp\begin{pmatrix}
0 & 0 & 0 & 0\\
x & 0 & t & 0\\
y & 0 & 0 & 0\\
w & 0 & 0 & 0
\end{pmatrix}\right)=\Phi\begin{pmatrix}
1 & 0 & 0 & 0\\
x+\frac{yt}2 & 1 & t & 0\\
y & 0 & 1 & 0\\
w & 0 & 0 & 1
\end{pmatrix}
$$
$$
=\exp\begin{pmatrix}
0 & 0 & 0 & 0\\
\left[\alpha\Delta_{22}-\beta_2\gamma_2\right]x+\Delta_{12}y+\gamma_1t+\phi_{01}(w) & 0 & \beta_2y+\alpha t & 0\\
\Delta_{22}y+\gamma_2t & 0 & 0 & 0\\
\eta y+\rho t+\phi_{11}(w) & 0 & 0 & 0
\end{pmatrix}
$$
$$
=\begin{pmatrix}
1 & 0 & 0 & 0\\
\left[\alpha\Delta_{22}-\beta_2\gamma_2\right]x+\Delta_{12}y+\gamma_1t+\phi_{01}(w)+\frac{\left[\Delta_{22}y+\gamma_2t\right]\left[\beta_2y+\alpha t\right]}2 & 1 & \beta_2y+\alpha t & 0\\
\Delta_{22}y+\gamma_2t & 0 & 1 & 0\\
\eta y+\rho t+\phi_{11}(w) & 0 & 0 & 1
\end{pmatrix},
$$
which yields the desired assertion. For the generic case let us first show that the map
$$
[v,t]\overset{\Phi}{\longmapsto}\left[\frac{e^{\alpha t\operatorname{J}(\aleph)}-\id}{\alpha\operatorname{J}(\aleph)}\gamma+\Delta v,\alpha t\right]
$$
is bijective by checking that its inverse is given by
$$
[v,t]\overset{\Phi^{-1}}{\longmapsto}\left[-\frac{e^{t\operatorname{J}(\aleph)}-\id}{\operatorname{J}(\aleph)}\Delta^{-1}\gamma+\Delta^{-1} v,\frac{t}\alpha\right].
$$
Indeed,
$$
\Phi^{-1}\circ\Phi[v,t]=\left[-\frac{e^{t\operatorname{J}(\aleph)}-\id}{\operatorname{J}(\aleph)}\Delta^{-1}\gamma+\Delta^{-1}\left[\frac{e^{\alpha t\operatorname{J}(\aleph)}-\id}{\alpha\operatorname{J}(\aleph)}\gamma+\Delta v\right],\frac{\alpha t}\alpha\right]=[v,t],
$$
where we used $\alpha\operatorname{J}(\aleph)\Delta=\Delta\operatorname{J}(\aleph)$ and Remark \ref{FABCRemark}. Next we establish that the same map is a Lie group homomorphism,
$$
\Phi[v,t]\cdot\Phi[u,s]=\left[\frac{e^{\alpha t\operatorname{J}(\aleph)}-\id}{\alpha\operatorname{J}(\aleph)}\gamma+\Delta v,\alpha t\right]\cdot\left[\frac{e^{\alpha s\operatorname{J}(\aleph)}-\id}{\alpha\operatorname{J}(\aleph)}\gamma+\Delta u,\alpha s\right]
$$
$$
=\left[\frac{e^{\alpha(t+s)\operatorname{J}(\aleph)}-\id}{\alpha\operatorname{J}(\aleph)}\gamma+\Delta\left[v+e^{t\operatorname{J}(\aleph)}u\right],\alpha(t+s)\right]=\Phi[v+e^{t\operatorname{J}(\aleph)}u,t+s]=\Phi\left([v,t]\cdot[u,s]\right).
$$
Finally, for every $(u,s)\in\mathbb{R}^d\rtimes\mathbb{R}=\aA(\aleph)$ let $(-1,1)\ni\tau\mapsto[v(\tau),t(\tau)]\in G$ be a smooth curve such that $[v(0),t(0)]=[0,0]$ and $(v'(0,t'(0))=(u,s)$. Then
$$
d\Phi(u,s)=\frac{d}{d\tau}\Phi[v(\tau),t(\tau)]|_{\tau=0}=\frac{d}{d\tau}\left[\frac{e^{\alpha t\operatorname{J}(\aleph)}-\id}{\alpha\operatorname{J}(\aleph)}\gamma+\Delta v(\tau),\alpha t(\tau)\right]\biggr\rvert_{\tau=0}=(\Delta u+s\gamma,\alpha s),
$$
which completes the proof. $\Box$
\end{proof}

\begin{remark}\label{HxRBeta3Remark} Again, if we apply formula (\ref{AutG}) to a central extension $G=\mathcal{H}\times\mathbb{R}^{d-2}$ of the Heisenberg group then we will exactly recover those automorphisms with $\beta_2=0$ in formula (\ref{AutHxR}).
\end{remark}

The normal subgroup $\Inn(G)\subset\Aut(G)$ of inner automorphisms contains $\Phi_g\in\Aut(G)$ such that $\Phi_g(h)=ghg^{-1}$ for some $g\in G$ and all $h\in G$.
\begin{corollary} If $G$ is a simply connected almost Abelian Lie group with Lie algebra $\aA(\aleph)$ then
$$
\Inn(G)=\left\{[v,t]\overset{\Phi}{\longmapsto}\left[\frac{e^{t\operatorname{J}(\aleph)}-\id}{\operatorname{J}(\aleph)}\gamma+\Delta v,t\right]\,\vline\quad\gamma\in\operatorname{J}(\aleph)\left(\mathbb{R}^d\right),\quad\Delta=e^{s\operatorname{J}(\aleph)},\quad s\in\mathbb{R}\right\}.
$$
\end{corollary}
\begin{proof} That $\Phi_g\in\Inn(G)$ means that $\Phi_g(h)=ghg^{-1}$, for $g\in G$, $\forall h\in G$. Let $g=[u,s]$ and $h=[v,t]$, so that
$$
\Phi_{[u,s]}[v,t]=[u,s][v,t][u,s]^{-1}=\left[e^{s\operatorname{J}(\aleph)}v-\left(e^{t\operatorname{J}(\aleph)}-\id\right)u,t\right]=\left[\frac{e^{t\operatorname{J}(\aleph)}-\id}{\operatorname{J}(\aleph)}\gamma+\Delta v,t\right],
$$
where
$$
\Delta=e^{s\operatorname{J}(\aleph)},\quad\gamma=-\operatorname{J}(\aleph)u,
$$
precisely as asserted. $\Box$
\end{proof}

We turn now to the case of more general connected almost Abelian Lie group $G/N$ where $G$ is simply connected and $N\subset G$ is a discrete central subgroup. Denote by $\operatorname{q}_N:G\to G/N$ the canonical quotient homomorphism. By Proposition \ref{G_N=G_M} we know that
$$
\Aut(G/N)=\left\{\Phi_N=\operatorname{q}_N\circ\Phi\circ\operatorname{q}_N^{-1}\,{\big|}\quad\Phi\in\Aut(G),\quad\Phi(N)=N\right\}.
$$
We will describe the condition $\Phi(N)=N$ more explicitly using Proposition \ref{AutGProp}. The following simple fact will come in handy.

\begin{lemma}\label{LintLemma} If $e^{t\operatorname{J}(\aleph)}=\id$ then $\aA(\aleph)=\mathbf{L}_0\oplus\mathbf{W}$ where $\mathbf{L}_0$ is indecomposable and $\mathbf{W}=\ker\operatorname{J}(\aleph)$, and
$$
\frac{e^{\alpha t\operatorname{J}(\aleph)}-\id}{\alpha\operatorname{J}(\aleph)}=t\left[0_{\mathbf{L}_0}\oplus\id_{\mathbf{W}}\right],\quad\forall\alpha\in\Dil(\aleph).
$$
\end{lemma}
\begin{proof} Note that
$$
\frac{e^{\alpha t\operatorname{J}(\aleph)}-\id}{\alpha\operatorname{J}(\aleph)}=t\frac{e^{\alpha t\operatorname{J}(\aleph)}-\id}{\alpha t\operatorname{J}(\aleph)}.
$$
If $t=0$ then
$$
\frac{e^{\alpha t\operatorname{J}(\aleph)}-\id}{\alpha t\operatorname{J}(\aleph)}=\id
$$
and the assertion is clear. If $t\neq0$ then $t\in T_\aleph\neq\emptyset$, and by Lemma \ref{TalephLemma}
$$
\frac{e^{\alpha t\operatorname{J}(\aleph)}-\id}{\alpha t\operatorname{J}(\aleph)}=\bigoplus_{p\in\supp\aleph}\bigoplus_{\aleph(p,1)}\frac{e^{\alpha t x_p}-\id}{\alpha t x_p}=\left[\bigoplus_{X\neq p\in\supp\aleph}\bigoplus_{\aleph(p,1)}\frac{e^{\alpha t x_p}-\id}{\alpha t x_p}\right]\,\bigoplus\,\left[\bigoplus_{\aleph(X,1)}\frac{e^{\alpha t x_p}-\id}{\alpha t x_p}\right]
$$
$$
=\left[0_{\mathbf{L}_0}\oplus\id_{\mathbf{W}}\right],
$$
as desired. $\Box$
\end{proof}

Now fix a central discrete subgroup $N\subset G$ and let by Proposition \ref{NProp} $N$ be generated by $\{[x_i,0,0,w_i]\}_{i=1}^k$ if $G=\mathcal{H}\times\mathbb{R}^{d-2}$ and $\{[v_i,t_i]\}_{i=1}^k$ otherwise.

\begin{proposition} In terminology of Proposition \ref{AutGProp}, an automorphism $\Phi\in\Aut(G)$ satisfies $\Phi(N)=N$ if and only if
\begin{equation}
\begin{pmatrix}
\alpha\Delta_{22}-\beta_2\gamma_2 & \phi_{01}\\
0 & \phi_{11}
\end{pmatrix}\begin{pmatrix}
x_1 & \ldots & x_k\\
w_1 & \ldots & w_k
\end{pmatrix}=\begin{pmatrix}
x_1 & \ldots & x_k\\
w_1 & \ldots & w_k
\end{pmatrix}\cdot A,\quad A\in\mathrm{GL}(\mathbb{Z},k)\label{PhiN=NHxR}
\end{equation}
for $G=\mathcal{H}\times\mathbb{R}^{d-2}$ and
\begin{equation}
\begin{pmatrix}
\Delta & \gamma_\mathbf{W}\\
0 & \alpha
\end{pmatrix}\begin{pmatrix}
v_1 & \ldots & v_k\\
t_1 & \ldots & t_k
\end{pmatrix}=\begin{pmatrix}
v_1 & \ldots & v_k\\
t_1 & \ldots & t_k
\end{pmatrix}\cdot A,\quad A\in\mathrm{GL}(\mathbb{Z},k)\label{PhiN=NG}
\end{equation}
otherwise. Here $\gamma_\mathbf{W}=\left[0_{\mathbf{L}_0}\oplus\id_{\mathbf{W}}\right]\gamma$ as per Lemma \ref{LintLemma}.
\end{proposition}
\begin{proof} If $G=\mathcal{H}\times\mathbb{R}^{d-2}$ then for every $\Phi\in\Aut(G)$ the condition $\Phi(N)\subset N$ can be expressed as the statement that for every fixed $1\le i_0\le k$, the image $\Phi([x_{i_0},0,0,w_{i_0}]$ is an integer linear combination of $\{[x_i,0,0,w_i]\}_{i=1}^k$. In matrix language of Proposition \ref{AutGProp} this can be written as
$$
\Phi\begin{bmatrix}
x_{i_0}\\
0\\
0\\
w_{i_0}
\end{bmatrix}=\begin{bmatrix}
\left[\alpha\Delta_{22}-\beta_2\gamma_2\right] x_{i_0}+\phi_{01}(w_{i_0})\\
0\\
0\\
\phi_{11}(w_{i_0})
\end{bmatrix}=\begin{pmatrix}
\alpha\Delta_{22}-\beta_2\gamma_2 & 0 & 0 & \phi_{01}\\
0 & 0 & 0 & 0\\
0 & 0 & 0 & 0\\
0 & 0 & 0 & \phi_{11}
\end{pmatrix}\begin{bmatrix}
x_{i_0}\\
0\\
0\\
w_{i_0}
\end{bmatrix}
$$
$$
=\begin{pmatrix}
x_1 & \ldots & x_k\\
0 & \ldots & 0\\
0 & \ldots & 0\\
w_1 & \ldots & w_k
\end{pmatrix}\begin{bmatrix}
A_{1\,i_0}\\
\ldots\\
A_{k\,i_0}
\end{bmatrix},\quad A_{i\,i_0}\in\mathbb{Z},\quad i=1,\ldots, k.
$$
Combining these statements for all $i_0=1,\ldots,k$ we obtain the formula (\ref{PhiN=NHxR}) with $A$ being a $k\times k$ matrix with integer entries. Following the same logic for $\Phi^{-1}(N)\subset N$ we will obtain a similar formula where the matrix $A^{-1}$ figures and is supposed to have integer coefficients. But $\Phi(N)=N$ is equivalent to $\Phi(N)\subset N$ and $\Phi^{-1}(N)\subset N$, which holds if and only if both $A$ and $A^{-1}$ have integer entries, i.e., $A\in\mathrm{GL}(\mathbb{Z},k)$, as desired. If $G\neq\mathcal{H}\times\mathbb{R}^{d-2}$ then by Proposition \ref{ZGProp} we see that $e^{t_i\operatorname{J}(\aleph)}=\id$ for all $i=1,\ldots,k$. Thus by Proposition \ref{AutGProp} and Lemma \ref{LintLemma} the condition $\Phi(N)\subset N$ becomes
$$
\Phi\begin{bmatrix}
v_{i_0}\\
t_{i_0}
\end{bmatrix}=\begin{bmatrix}
\frac{e^{\alpha t_{i_0}\operatorname{J}(\aleph)}}{\alpha\operatorname{J}}\gamma+\Delta v_{i_0}\\
\alpha t_{i_0}
\end{bmatrix}=\begin{pmatrix}
\Delta & \gamma_\mathbf{W}\\
0 & \alpha
\end{pmatrix}\begin{bmatrix}
v_{i_0}\\
t_{i_0}
\end{bmatrix}=\begin{pmatrix}
v_1 & \ldots & v_k\\
t_1 & \ldots & t_k
\end{pmatrix}\begin{bmatrix}
A_{1\,i_0}\\
\ldots\\
A_{k\,i_0}
\end{bmatrix},\quad A_{i\,i_0}\in\mathbb{Z},\quad i=1,\ldots, k.
$$
Combining these statements for all $i_0=1,\ldots,k$ we obtain the formula (\ref{PhiN=NG}) with $A$ being a $k\times k$ matrix with integer entries. The rest of the argument follows as before. $\Box$
\end{proof}

\begin{remark} Let $N\subset G$ be a discrete central subgroup. Since all $\Phi\in\Inn(G)$ act trivially on $N\subset\mathrm{Z}(G)$, it follows that $\Phi(N)=N$ is satisfied automatically.
\end{remark}

%\end{document}

\quad

%\documentclass{article}

%\input{include.tex}

%\begin{document}

\section*{Discrete normal subgroups and quotients of simply connected almost Abelian groups revisited}\label{DNSubGr2}
%\subsubsection*{Z. Avetisyan, I. Martin, Z. Zhao}

Pursuant to the aims of Proposition \ref{G_N=G_M}, in this section we want to derive necessary and sufficient conditions for two discrete central subgroups $N,M\subset G$ to be related by an automorphism $\Phi\in\Aut(G)$ of the simply connected almost Abelian Lie group $G$. We begin with preparatory steps with a discrete central subgroup $N\subset G$ given in terms of a set of generators $[v_1,t_1],\ldots,[v_k,t_k]$ according to Proposition \ref{NProp}. Every other set of generators $[u_1,s_1],\ldots,[u_k,s_k]$ of $N$ is related to the original one by
$$
\begin{pmatrix}
u_1\ldots u_k\\
s_1\ldots s_k
\end{pmatrix}=\begin{pmatrix}
v_1\ldots v_k\\
t_1\ldots t_k
\end{pmatrix}\,\cdot\,A,\quad A\in\mathrm{GL}(\mathbb{Z},k).
$$
According to Lemma \ref{TalephLemma}, there exists $t_0\in T_\aleph$ and $n_1,\ldots,n_k\in\mathbb{Z}$ such that $t_i=n_it_0$, $i=1,\ldots,k$.
\begin{lemma}\label{GenReduceLemma} There exists a change of generators $A\in\mathrm{GL}(\mathbb{Z},k)$ such that
$$
\begin{pmatrix}
u_1 & u_2 & \ldots & u_k\\
s_1 & 0 & \ldots & 0
\end{pmatrix}=\begin{pmatrix}
v_1 & v_2 & \ldots & v_k\\
t_1 & t_2 & \ldots & t_k
\end{pmatrix}\,\cdot\, A.
$$
\end{lemma}
\begin{proof}This can be achieved easily by column operations justified with Bezout's identity. See Appendix \ref{Appendix}. $\Box$
\end{proof}
In what follows we will assume that a discrete central subgroup $N\subset G$ is given by a set of generators in the more economic form $[v_1,t_1],[v_2,0],\ldots,[v_k,0]$. In terminology of formula (16) in \cite{Avetisyan2018},
$$
\ker\operatorname{J}(\aleph)=\bigoplus_{n=1}^\infty\bigoplus_{\aleph(X,n)}\mathbb{R}e^1_\alpha(X,n),
$$
or in other words, the vectors $v_1,\ldots,v_k\in\mathbb{R}^d$ written in the standard basis ${e^m_\alpha(p,n)}$ may have non-zero entries only in the rows corresponding to the topmost elements of the Jordan blocks with eigenvalue zero. Let $\tilde v_1,\ldots,\tilde v_k\in\mathbb{R}^q$, $q\doteq\dim\ker\operatorname{J}(\aleph)$, be the vectors obtained by picking only these significant rows. We have seen in Proposition \ref{AutGProp} that operators $\Delta\in\Aut(\mathbb{R}^d)$ with $[\Delta,\operatorname{J}(\aleph)]=0$ play a prominent role in the structure of automorphisms of $G$. Such an operator $\Delta$ preserves the invariant subspace $\ker\operatorname{J}(\aleph)$, and we denote the restriction of $\Delta$ to $\ker\operatorname{J}(\aleph)$ by $\tilde\Delta\in\Aut(\mathbb{R}^q)$. Let us now assume that Jordan blocks in $\operatorname{J}(\aleph)$ are ordered by non-decreasing block dimension $n$. Applying Proposition 7 and Lemma 2 from \cite{Avetisyan2018}, we see that $\Delta=\tilde\Delta\oplus 0$ (i.e., the matrix $\Delta$ beyond the submatrix $\tilde\Delta$ is identically zero) and $\tilde\Delta$ is an arbitrary real invertible block-upper-triangular matrix with blocks corresponding to constant Jordan block dimension $n$. That means,
$$
\tilde\Delta=\begin{pmatrix}
\tilde\Delta_{n_1n_1} & \tilde\Delta_{n_1n_2} & \ldots & \tilde\Delta_{n_1n_s}\\
0 & \tilde\Delta_{n_2n_2} & \ldots & \tilde\Delta_{n_2n_s}\\
\ldots & \ldots & \ldots & \ldots\\
0 & 0 & \ldots & \tilde\Delta_{n_sn_s}
\end{pmatrix},\quad\tilde\Delta_{n_in_j}\in\Hom(\mathbb{R}^{q_j},\mathbb{R}^{q_i}),\quad\aleph(X,n_i)=q_i,\quad i,j=1,\ldots,s,
$$
$$
q_1+\ldots+q_s=q,\quad n_i>n_{i+1},\quad i=1,\ldots,s-1.
$$
The following simple observation will be useful in what follows.
\begin{remark}\label{TalephRemark} In terminology of \cite{Avetisyan2018}, $\mathrm{Dil}(\aleph)\subset\mathbb{R}^*$ is a finite multiplicative subgroup and therefore $\mathrm{Dil}(\aleph)\subset\mathbb{Z}_2$. If $\supp\aleph\subset\imath\mathbb{R}$, which by Lemma \ref{TalephLemma} is the case when $T_\aleph\neq\{0\}$, then necessarily $\mathrm{Dil}(\aleph)=\mathbb{Z}_2$.
\end{remark}

\begin{proposition} Two discrete central subgroups $N$ and $M$ given in terms of generators $[v_1,t_1],[v_2,0],\ldots,[v_k,0]$ and $[u_1,s_1],[u_2,0],\ldots,[u_k,0]$, respectively, are related by an automorphism of $G$ if and only if $t_1=\pm s_1$ and there exist $\tilde\Delta$ as above and an $A\in\mathrm{GL}(\mathbb{Z},k)$ such that
$$
\tilde\Delta\cdot\left(
\tilde v_1\,\tilde v_2\,\ldots\,\tilde v_k
\right)=\left(
\tilde u_1\,\tilde u_2\,\ldots\,\tilde u_k
\right)\cdot A\quad\mbox{if}\quad t_1=0
$$
and
$$
\tilde\Delta\cdot\left(
\tilde w\,\tilde v_2\,\ldots\,\tilde v_k
\right)=\left(
\tilde u_1\,\tilde u_2\,\ldots\,\tilde u_k
\right)\cdot A\quad\mbox{if}\quad t_1\neq0,
$$
where $\tilde w\in\mathbb{R}^q$ can be chosen arbitrarily.
\end{proposition}
\begin{proof} The subgroup $N$ is mapped to the subgroup $M$ by an automorphism $\Phi\in\Aut(G)$ if and only if the generators $[v_1,t_1],[v_2,0],\ldots,[v_k,0]$ are mapped to any set of generators of $M$, which must be related to the original generators $[u_1,s_1],[u_2,0],\ldots,[u_k,0]$ through a matrix $A\in\mathrm{GL}(\mathbb{Z},k)$, i.e.,
$$
\Phi\begin{pmatrix}
v_1 & v_2 & \ldots & v_k\\
t_1 & 0 & \ldots & 0
\end{pmatrix}=\begin{pmatrix}
u_1 & u_2 & \ldots & u_k\\
s_1 & 0 & \ldots & 0
\end{pmatrix}\cdot A.
$$
By Proposition \ref{AutGProp} this amounts to
$$
\begin{pmatrix}
\Delta & \gamma\\
0 & \alpha
\end{pmatrix}\cdot\begin{pmatrix}
v_1 & v_2 & \ldots & v_k\\
t_1 & 0 & \ldots & 0
\end{pmatrix}=\begin{pmatrix}
u_1 & u_2 & \ldots & u_k\\
s_1 & 0 & \ldots & 0
\end{pmatrix}\cdot A,
$$
since even for $\mathcal{H}\times\mathbb{R}^{d-2}$ the coefficient $\beta_2$ has no effect in acting on vectors from $\ker\operatorname{J}(\aleph)$. By Remark \ref{TalephRemark} we have $\alpha=\pm1$ so that $t_1=\pm s_1$. Further,
$$
\Delta\cdot(v_1,\,v_2,\,\ldots\,v_k)+(t_1\gamma,\,0,\,\ldots,\,0)=(v_1,\,v_2,\,\ldots\,v_k)\cdot A,
$$
where the choice of $\gamma\in\mathbb{R}^d$ is completely arbitrary. The assertion now follows by restricting the above equation to $\ker\operatorname{J}(\aleph)$. $\Box$
\end{proof}
Finding algebraic criteria under which the above conditions are satisfied is a hard problem which we will not pursue here.

As a simple side result, the structure of a discrete central subgroup $N\subset G$ can be simplified further using automorphisms. In the above economic form of the basis for $N$ the element $v_1$ is arbitrary, and it need not be possible to kill $v_1$ by any further right $\mathrm{GL}(\mathbb{Z},k)$ action. Instead, we can use automorphisms of $G$ to achieve that simplification.
\begin{proposition}\label{GenFReduceProp} For every discrete central subgroup $N\subset G$ of a simply connected almost Abelian group $G=\mathbb{R}^d\rtimes\mathbb{R}$ with Lie algebra $\aA(\aleph)$ there exists an automorphism $\Phi\in\Aut(G)$ such that the discrete central subgroup $M=\Phi(N)$ satisfies $M=(M\cap\ker\operatorname{J}(\aleph))\times(M\cap T_\aleph)$.
\end{proposition}
\begin{proof} Let $N$ be given in terms of the generators $[v_1,t_1],[v_2,0],\ldots,[v_k,0]$. If $t_1=0$ then $N\subset\ker\operatorname{J}(\aleph)$ and the assertion is trivial. Assume that $t_1\neq0$, so that by Lemma \ref{TalephLemma} we have $\supp\aleph\subset\imath\mathbb{R}$, and therefore $\mathrm{Dil}(\aleph)=\mathbb{Z}_2$ \cite{Avetisyan2018}. Choose $\Phi$ according to Proposition \ref{AutGProp} with $\alpha=\mathrm{sgn}\,t_1$, $\Delta=\id$ and $\gamma=-\frac1{t_1}v_1$. Then $M=\Phi(N)$ is given by the set of generators
$$
\begin{pmatrix}
\id & -\frac1{t_1}v_1\\
0 & \mathrm{sgn}\,t_1
\end{pmatrix}\cdot\begin{pmatrix}
v_1 & v_2 & \ldots & v_k\\
t_1 & 0 & \ldots & 0
\end{pmatrix}=\begin{pmatrix}
0 & v_2 & \ldots & v_k\\
|t_1| & 0 & \ldots & 0
\end{pmatrix},
$$
whence the statement of the proposition follows. $\Box$
\end{proof}

%\end{document}

\quad

%\documentclass{article}

%\input{include.tex}

%\begin{document}

\section*{Connected almost Abelian groups}\label{CG}
%\subsubsection*{Z. Avetisyan, I. Martin, Z. Zhao}

The goal of this section is to describe connected (not necessarily simply connected) almost Abelian groups in terms of faithful matrix representations whenever the latter exist. Recall that a connected almost Abelian group can be written as $G/N$ where the universal cover $G$ is a simply connected almost Abelian group and $N\subset G$ is a discrete central subgroup. Regardless of whether $G/N$ is a matrix group, the matrix representation of $G$ can be used to produce a natural (almost global) coordinate chart on $G/N$ as follows. Consider a modification of the second faithful matrix representation of $G$ from Proposition \ref{SCMatRepProp} as a faithful "quotient-matrix" representation of $G/N$,
$$
G/N\ni[v,t]\mod N\quad\mapsto\quad\begin{pmatrix}
\begin{matrix}
1\\
\begin{bmatrix}
v\\
t
\end{bmatrix}_{\!\!\!\!\!\!\mod N}
\end{matrix} & \begin{matrix}
0 & 0\\
e^{t\operatorname{J}(\aleph)} & 0\\
0 & 1
\end{matrix}
\end{pmatrix}\in\End(\mathbb{R}^{d+2}).
$$
This representation is algebraically convenient since by Proposition \ref{NProp} we know that $N$ can be seen as an additive subrgoup of $\mathbb{R}^{d+1}$, and $[v,t]\mod N$ is easy to compute. In a neighbourhood of the identity the above representation coincides with the true faithful matrix representation of $G$.

Let us now turn to proper faithful matrix representations. The following provides an explicit faithful matrix representation for a quotient group $G/N$ under certain assumptions on $N$. Let
$$
\aA(\aleph)=\mathbb{R}^{d_0}\rtimes\mathbb{R}\oplus\mathbb{R}^{d-d_0}
$$
be a decomposition of $\aA(\aleph)$ as in [Ave16] where $\mathbb{R}^{d_0}\rtimes\mathbb{R}$ is indecomposable. Then the simply connected group decomposes as $G=G_0\times\mathbb{R}^{d-d_0}$. The first faithful representation of $G$ from Proposition \ref{SCMatRepProp}, upon substitution of the decomposition $\mathbb{R}^d\ni u\mapsto v\oplus w\in\mathbb{R}^{d_0}\oplus\mathbb{R}^{d-d_0}$, gives
$$
G=\mathbb{R}^{d_0}\rtimes\mathbb{R}\times\mathbb{R}^{d-d_0}\ni[v,t,w]\quad\mapsto\quad\begin{pmatrix}
1 & 0 & 0 & 0\\
v & e^{t\operatorname{J}(\aleph_0)} & 0 & 0\\
w & 0 & \id & 0\\
0 & 0 & 0 & e^t
\end{pmatrix}\in\End(\mathbb{R}^{d+1}).
$$
If we denote by $\diag w$ the $(d-d_0)$-dimensional diagonal matrix composed of components of $w$ then it can be easily checked that
\begin{equation}
\begin{pmatrix}
1 & 0 & 0 & 0\\
v & e^{t\operatorname{J}(\aleph_0)} & 0 & 0\\
w & 0 & \id & 0\\
0 & 0 & 0 & e^t
\end{pmatrix}\quad\mapsto\quad\begin{pmatrix}
1 & 0 & 0 & 0\\
v & e^{t\operatorname{J}(\aleph_0)} & 0 & 0\\
0 & 0 & e^{\diag w} & 0\\
0 & 0 & 0 & e^t
\end{pmatrix}\label{DecompFaithRep}
\end{equation}
is a matrix Lie group isomorphism, and therefore the right hand side is another faithful matrix representation of $G$.

Assume now that the discrete central subgroup satisfies $N\subset\mathbb{R}^{d-d_0}\times T_\aleph$, i.e., per Proposition \ref{NProp}, is generated by
$$
[w_1,t_1],\ldots,[w_k,t_k]\in\mathbb{R}^{d-d_0}\times T_\aleph,\quad 0\le k\le d-d_0+1.
$$
The representation on the right hand side of (\ref{DecompFaithRep}) is convenient in that it allows to reshuffle the last $d-d_0+1$ dimensions in way to separate the generators of $N$. Namely, complete arbitrarily the above generators of $N$ to a basis in $\mathbb{R}^{d-d_0}\oplus\mathbb{R}$,
$$
[w_1,t_1],\ldots,[w_{d-d_0+1},t_{d-d_0+1}]\in\mathbb{R}^{d-d_0}\oplus\mathbb{R},
$$
and consider the inverse $\operatorname{P}\in\End(\mathbb{R}^{d-d_0+1})$ of the matrix with columns being elements of this basis,
\begin{equation}
\operatorname{P}\doteq\begin{pmatrix}
w_1 & \ldots & w_{d-d_0+1}\\
t_1 & \ldots & t_{d-d_0+1}
\end{pmatrix}^{-1}.\label{PDef}
\end{equation}
Let $\operatorname{P}_\parallel\in\Hom(\mathbb{R}^{d-d_0+1},\mathbb{R}^k)$ represent the first $k$ rows of $\operatorname{P}$, and $\operatorname{P}_\perp\in\Hom(\mathbb{R}^{d-d_0+1},\mathbb{R}^{d-d_0+1-k})$ the remaining rows.

\begin{proposition}\label{FaithfulRepProp} If the discrete central subgroup satisfies $N\subset\mathbb{R}^{d-d_0}\times T_\aleph$ then the map
$$
G/N\ni[v,t,w]\mod N\quad\mapsto\quad\begin{pmatrix}
1 & 0 & 0 & 0\\
v & e^{t\operatorname{J}(\aleph_0)} & 0 & 0\\
0 & 0 & e^{\diag2\pi\imath\operatorname{P}_\parallel[w,t]^\top} & 0\\
0 & 0 & 0 & e^{\diag\operatorname{P}_\perp[w,t]^\top}
\end{pmatrix}\quad\in\End(\mathbb{R}^{d+2})
$$
is a faithful matrix representation of $G/N$.
\end{proposition}
\begin{proof} In view of (\ref{DecompFaithRep}) being a faithful representation of $G$, it suffices to show that the map
$$
\begin{pmatrix}
1 & 0 & 0 & 0\\
v & e^{t\operatorname{J}(\aleph_0)} & 0 & 0\\
0 & 0 & e^{\diag w} & 0\\
0 & 0 & 0 & e^t
\end{pmatrix}\quad\mapsto\quad\begin{pmatrix}
1 & 0 & 0 & 0\\
v & e^{t\operatorname{J}(\aleph_0)} & 0 & 0\\
0 & 0 & e^{\diag2\pi\imath\operatorname{P}_\parallel[w,t]^\top} & 0\\
0 & 0 & 0 & e^{\diag\operatorname{P}_\perp[w,t]^\top}
\end{pmatrix}
$$
is a Lie group homomorphism with kernel $N$. Checking that this is a Lie group homomorphism is straightforward. Now let $[v,t,w]\in\mathbb{R}^{d_0}\rtimes\mathbb{R}\times\mathbb{R}^{d-d_0}=G$. Then
$$
\begin{pmatrix}
1 & 0 & 0 & 0\\
v & e^{t\operatorname{J}(\aleph_0)} & 0 & 0\\
0 & 0 & e^{\diag2\pi\imath\operatorname{P}_\parallel[w,t]^\top} & 0\\
0 & 0 & 0 & e^{\diag\operatorname{P}_\perp[w,t]^\top}
\end{pmatrix}=\begin{pmatrix}
1 & 0 & 0 & 0\\
0 & \id & 0 & 0\\
0 & 0 & \id & 0\\
0 & 0 & 0 & \id
\end{pmatrix}
$$
iff
$$
v=0,\quad t\in T_\aleph,\quad\operatorname{P}_\parallel[w,t]^\top\in\mathbb{Z}^k,\quad\operatorname{P}_\perp[w,t]^\top=0.
$$
The latter two conditions can be combined into $\operatorname{P}[w,t]^\top\in\mathbb{Z}^k\oplus0$, which in view of (\ref{PDef}) can be written as
$$
\begin{pmatrix}
w\\
t
\end{pmatrix}=\begin{pmatrix}
w_1 & \ldots & w_{d-d_0+1}\\
t_1 & \ldots & t_{d-d_0+1}
\end{pmatrix}\begin{pmatrix}
m\\
0
\end{pmatrix},\quad m\in\mathbb{Z}^k,
$$
which is equivalent to $[w,t]$ being generated by $[w_1,t_1],\ldots,[w_k,t_k]$ over $\mathbb{Z}$. Thus $[v,t,w]$ is in the kernel iff $[v,t,w]\in N$, which completes the proof. $\Box$
\end{proof}

Below we establish a necessary and sufficient condition for $G/N$ to be a matrix group in terms of the subgroup $N$. We start with a little lemma.

\begin{lemma}\label{AntiSFHeisLemma} Let $X,Y,Z\in\End(\mathbb{C}^n)$ be such that
$$
[X,Y]=Z,\quad[X,Z]=[Y,Z]=0,\quad Z+Z^*=0.
$$
Then $Z=0$.
\end{lemma}
\begin{proof} Since $Z$ is anti-Hermitean, by the spectral theorem for Hermitean matrices it is unitarily diagonalizable with purely imaginary spectrum. Assume without loss of generality that
$$
Z=\imath\bigoplus_{i=1}^q\lambda_i\id_{n_i},\quad\lambda_i\in\mathbb{R},\quad n_1+\ldots+n_q=n.
$$
Then by Proposition 7 in [Ave18] the matrices $X$ and $Y$ are of the form
$$
X=\bigoplus_{i=1}^qX_i,\quad Y=\bigoplus_{i=1}^qY_i,\quad X_i,Y_i\in\End(\mathbb{C}^{n_i}).
$$
Thus $[X_i,Y_i]=\lambda_i\id_{n_i}$ and therefore $\tr[X_i,Y_i]=0=\lambda_i$, $i=1,\ldots,q$, which shows that $Z=0$. $\Box$
\end{proof}

\begin{proposition} Let $G=\mathbb{R}^d\rtimes\mathbb{R}$ be a simply connected almost Abelian group with Lie algebra $\mathbf{L}=\mathbb{R}^d\rtimes\mathbb{R}=\aA(\aleph)$, and let $N\subset G$ be a discrete central subgroup with generators $[v_1,t_1],\ldots,[v_k,t_k]$. Then the following two statements are equivalent:
\begin{itemize}

\item[1.] $\mathbb{R}\left\{(v_1,t_1),\ldots,(v_k,t_k)\right\}\cap[\mathbf{L},\mathbf{L}]=0$

\item[2.] $G/N$ has a faithful (real or complex) matrix representation

\end{itemize}
\end{proposition}
\begin{proof}\textit{2.\,$\Rightarrow$\,1.} Assume towards a contradiction that condition 1. is not satisfied,
$$
[(0,1),(u,0)]=\sum_{i=1}^k\lambda_i(v_i,t_i)\neq0,\quad u\in\mathbb{R}^d,\quad\lambda_i\in\mathbb{R},\quad i=1,\ldots, k,
$$
and let $\sigma:G/N\to\Aut(\mathbb{C}^n)$ be a faithful representation with $d\sigma:\mathbf{L}\to\End(\mathbb{C}^n)$ being its derivative. Because $\exp|_{\mathrm{Z}(\mathbf{L})}=\id$ we have that $\exp(v_i,t_i)=[v_i,t_i]$ and thus $\sigma[v_i,t_i]=e^{d\sigma(v_i,t_i)}=\id$, which implies by Lemma \ref{TalephLemma} that $d\sigma(v_i,t_i)$ is diagonalizable with spectrum in $2\pi\imath\mathbb{Z}$. Moreover, since $[d\sigma(v_i,t_i),d\sigma(v_j,t_j)]=0$ for all $i,j=1,\ldots, k$, there is an invertible $P\in\Aut(\mathbb{C}^n)$ such that
$$
P^{-1}d\sigma(v_i,t_i)P=\imath D_i,\quad D_i^*=D_i,\quad i=1,\ldots,k.
$$
Denote
$$
X\doteq P^{-1}d\sigma(0,1)P,\quad Y\doteq P^{-1}d\sigma(0,u)P,\quad Z\doteq\imath\sum_{i=1}^kD_i.
$$
Then the assumptions of Lemma \ref{AntiSFHeisLemma} are satisfied, implying that
$$
Z=d\sigma([(0,1),(u,0)])=0,\quad [(0,1),(u,0)]\neq0,
$$
which contradicts the fact that $\sigma$ is faithful.

\textit{1.\,$\Rightarrow$\,2.} Let now condition 1. be satisfied. By Lemma \ref{GenReduceLemma} we can assume without loss of generality that $t_2=t_3=\ldots=t_k=0$. If $\mathbf{L}=\mathbb{R}^{d_0}\rtimes\mathbb{R}\oplus\mathbb{R}^{d-d_0}$ is the decomposition as before then condition 1. implies that $v_2,\ldots,v_k\in\mathbb{R}^{d-d_0}$. If $t_1=0$ then condition 1. also requires that $v_1\in\mathbb{R}^{d-d_0}$, which shows that $N\subset\mathbb{R}^{d-d_0}$, and by Proposition \ref{FaithfulRepProp} the quotient group $G/N$ has a faithful matrix representation. If $t_1\neq0$ then applying the automorphism $\Phi\in\Aut(G)$ from Proposition \ref{GenFReduceProp} we obtain the discrete central subgroup $\Phi(N)$ with generators $[0,t_1],[v_2,0],\ldots,[v_k,0]$, which now satisfies $\Phi(N)\subset\mathbb{R}^{d-d_0}\times T_\aleph$. Thus by Proposition \ref{FaithfulRepProp} the quotient group $G/\Phi(N)$ has a faithful matrix representation. But then by Proposition \ref{G_N=G_M} the automorphism $\Phi$ induces an isomorphism between $G/N$ and $G/\Phi(N)$, proving that $G/N$ has a faithful matrix representation, too. $\Box$
\end{proof}

%\end{document}

\quad

%\documentclass{article}

%\input{include.tex}

%\begin{document}

\section*{Connected subgroups of a connected almost Abelian Lie group}\label{CSubGr}
%\subsubsection*{Z. Avetisyan, I. Martin, G. Rakholia, Z. Zhao}

The goal of this section is describing all connected Lie subgroups of a connected almost Abelian Lie group. A connected almost Abelian group can be identified with the quotient group $G/N$ where $G$ is a simply connected almost Abelian Lie group and $N\subset G$ is a discrete normal subgroup (see Proposition \ref{NProp}). The canonical quotient map $\operatorname{q}_N:G\to G/N$ is a Lie group homomorphism, and its derivative $d\operatorname{q}_N$ is an isomorphism of Lie algebras. Thus we can assume without loss of generality that the Lie algebras of both $G$ and $G/N$ are $\aA(\aleph)$. By the Lattice Isomorphism Theorem (Theorem 20 in \cite{DummitFoote2004}) subgroups $H_N\subset G/N$ are exactly the quotients $H/N$ of subgroups $H\subset G$ with $N\subset H\subset G$. However, the complete preimage $\operatorname{q}_N(H_N)\subset G$ may not be a closed subgroup, and we may have to choose a different $H$ with $H/N=H_N$.

We will start from a simply connected almost Abelian Lie group $G$ with Lie algebra $\aA(\aleph)=\mathbb{R}^d\rtimes\mathbb{R}$. By Theorem 5.20 in \cite{Hall2015} to every Lie subalgebra $\mathbf{L}\subset\aA(\aleph)$ there exists a unique connected Lie subgroup $H_\mathbf{L}\subset G$ for which it is the Lie algebra, and conversely, all connected Lie subgroups of $G$ arise in this way.

\begin{remark}\label{LCasesRemark} By Proposition 4 in \cite{Avetisyan2016}, either of the following two possibilities occurs:
\begin{itemize}

\item[1.] $\mathbf{L}=\mathbf{W}\subset\mathbb{R}^d$ is an Abelian Lie subalgebra.

\item[2.] $\mathbf{L}$ is of the form
$$
\mathbf{L}=\left\{(w+tv_0,t)\in\mathbb{R}^d\rtimes\mathbb{R}\,{\big|}\quad w\in\mathbf{W},\quad t\in\mathbb{R}\right\},
$$
where $v_0\in\mathbb{R}^d$ is a fixed element and $\mathbf{W}\subset\mathbb{R}^d$ is an $\ad$-invariant vector subspace. In this case $\mathbf{L}$ is Abelian if and only if $\mathbf{W}\subset\mathrm{Z}(\aA(\aleph))$.

\end{itemize}
\end{remark}
Accordingly, the corresponding connected Lie subgroups $H_\mathbf{L}$ fall into two categories.

\begin{proposition}\label{CSubGrSCGProp} The connected Lie subgroup $H_\mathbf{L}\subset G$ of the simply connected almost Abelian Lie group $G$ with Lie algebra $\mathbf{L}$ as in Remark \ref{LCasesRemark} is given by either of the following two forms, accordingly:
\begin{itemize}

\item[1.]
$$
H_\mathbf{L}=\left\{[w,0]\in\mathbb{R}^d\rtimes\mathbb{R}\,{\big|}\quad w\in\mathbf{W}\right\}=\exp(\mathbf{W})
$$

\item[2.]
$$
H_\mathbf{L}=\left\{\left[w+\frac{e^{t\operatorname{J}(\aleph)}-\id}{\operatorname{J}(\aleph)}v_0,t\right]\in\mathbb{R}^d\rtimes\mathbb{R}\,{\big|}\quad w\in\mathbf{W},\quad t\in\mathbb{R}\right\}\simeq\exp(\mathbf{W})\cdot\mathbb{R}
$$

\end{itemize}
In the second case
$$
\exp(\mathbf{W})\cdot\mathbb{R}=\begin{cases}
\exp(\mathbf{W})\times\mathbb{R}\quad\mbox{if}\quad\mathbf{W}\subset\mathrm{Z}(\aA(\aleph)),\\
\exp(\mathbf{W})\rtimes\mathbb{R}\quad\mbox{else}.
\end{cases}
$$
\end{proposition}
\begin{proof} That $H_\mathbf{L}$ is indeed a Lie subgroup in both cases can be checked directly using, say, the faithful matrix representation $\mathrm{I}$ of Proposition \ref{SCMatRepProp}. In Case 1 the exponential map from Lemma \ref{ExpMapLemma} delivers the desired result immediately. For Case 2, pick an arbitrary $(w_0+t_0v_0,t_0)\in\mathbf{L}$ and let $(-1,1)\ni\tau\mapsto(w(\tau),t(\tau))\in\mathbf{W}\oplus\mathbb{R}$ be a smooth curve with
$$
(w(0),t(0))=(0,0),\quad(w'(0),t'(0))=(w_0,t_0)\in\mathbf{W}\oplus\mathbb{R}.
$$
Then we have
$$
\frac{d}{d\tau}\left[w(\tau)+\frac{e^{t(\tau)\operatorname{J}(\aleph)}-\id}{\operatorname{J}(\aleph)}v_0,t(\tau)\right]\biggr\rvert_{\tau=0}=(w_0+t_0v_0,t_0),
$$
showing that the Lie algebra of $H_\mathbf{L}$ is $\mathbf{L}$. Finally, an automorphism with $\alpha=1$, $\Delta=\id$ and $\gamma=v_0$ from Proposition \ref{AutGProp} can be used to establish the isomorphism between $H_\mathbf{L}$ and $\exp(\mathbf{W})\cdot\mathbb{R}$. $\Box$
\end{proof}

\begin{remark} Proposition \ref{CSubGrSCGProp} easily implies, in particular, that all connected subgroups of a simply connected almost Abelian group are simply connected and closed.
\end{remark}

\begin{remark} By Proposition 11 in \cite{Avetisyan2016}, two almost Abelian Lie subalgebras $\mathbf{L}_1,\mathbf{L}_2\subset\aA(\aleph)$ corresponding to $\ad$-invariant vector subspaces $\mathbf{W}_1,\mathbf{W}_2\subset\mathbb{R}^d$ are isomorphic if and only if $\operatorname{J}(\aleph)|_{\mathbf{W}_1}$ and $\operatorname{J}(\aleph)|_{\mathbf{W}_1}$ are projectively similar. Since both $H_{\mathbf{L}_1}$ and $H_{\mathbf{L}_2}$ are simply connected, we have that $H_{\mathbf{L}_1}\simeq H_{\mathbf{L}_1}$ if and only if $\mathbf{L}_1\simeq\mathbf{L}_2$.
\end{remark}

\begin{remark} By Corollary 5.7 in \cite{Hall2015}, two connected subgroups $H_{\mathbf{L}_1},H_{\mathbf{L}_2}\subset G$ of a simply connected almost Abelian group $G$, associated with Lie algebras $\mathbf{L}_1,\mathbf{L}_2\subset\aA(\aleph)$, respectively, are related by an automorphism $\Phi\in\Aut(G)$ if and only if the Lie algebras are related by the automorphism $d\Pi\in\Aut(\aA(\aleph))$.
\end{remark}

Let us now consider subgroups $H_N\subset G/N$ of connected almost Abelian groups $G/N$.

\begin{lemma} Let $G$ be a Lie group and $N\subset G$ a normal subgroup. Then every connected subgroup $H_N\subset G/N$ is the projection $H_N=H/N$ of a unique connected Lie subgroup $H\subset G$.
\end{lemma}
\begin{proof} The quotient map $\operatorname{q}_N:G\to G/N$ is a surjective Lie group homomorphism, and its derivative $d\operatorname{q}_N:\mathbf{L}_G:\to\mathbf{L}_{G/N}$ is a surjective Lie algebra homomorphism. The preimage $d\operatorname{q}^{-1}\mathbf{L}_{H_N}$ of the Lie algebra of $H_N$ is a Lie subalgebra of $\mathbf{L}_G$, and thus is the Lie algebra of a unique connected subgroup $H\subset G$ (Theorem 5.20 in \cite{Hall2015} or Proposition 5.6.5 in \cite{RudolphSchmidt2013}). The image $\operatorname{q}_N(H)\subset G/N$ is a connected subgroup with Lie algebra $\mathbf{L}_{H_N}$, which by uniqueness must be $\operatorname{q}_N(H)=H_N$. Finally, if $H'\subset G$ is another connected subgroup with $\operatorname{q}_N(H')=H_N$ then $\mathbf{L}_{H'}=\mathbf{L}_H$, so that again by uniqueness $H'=H$. $\Box$
\end{proof}

\begin{remark} Since the projection $H/N$ of a connected subgroup $H\subset G$ is a connected subgroup $H/N\subset G$, we conclude that connected subgroups of $G/N$ are exactly images $H/N$ of connected subgroups $H\subset G$, which were already classified above.
\end{remark}

It remains to find when a given connected subgroup $H/N\subset G/N$ is closed. For this purpose we will first establish a simple fact regarding the relative structure of $H$ and $N$.

\begin{lemma}\label{NBLemma} Let $G$ be a simply connected almost Abelian group, $N\subset G$ a discrete normal subgroup and $H\subset G$ a connected subgroup. Then there exists a subgroup $B\subset N$ such that $N=(N\cap H)\times B$.
\end{lemma}
\begin{proof} We use Proposition \ref{CSubGrSCGProp} to write $H$ in the form $H=\exp(\mathbf{W})$ or $H=\exp(\mathbf{W})\rtimes\mathbb{R}$ (direct or semidirect), with $\mathbf{W}\subset\mathbb{R}^d$ a vector subspace. All we need to show is that the $N\cap H\subset N$ is a pure subgroup. Indeed, let $[v,t]\in N$ and $q\in\mathbb{N}$ such that $[v,t]^q=[qv,qt]\in N\cap H$. Then $qv\in\mathbf{W}$ and thus also $v\in\mathbf{W}$, whence $[v,t]\in N\cap H$. Then by Corollary 28.5 in \cite{Fuchs1970}, $N\cap H$ is a direct factor. $\Box$
\end{proof}

Since
$$
\exp|_{\mathrm{Z}(\aA(\aleph))\oplus\mathbb{R}}:\mathrm{Z}(\aA(\aleph))\oplus\mathbb{R}\to\mathrm{Z}(G)_0\times\mathbb{R}
$$
is a bijection, we can introduce its inverse
$$
\log=[\exp|_{\ker\operatorname{J}(\aleph)\oplus\mathbb{R}}]^{-1}:\mathrm{Z}(G)_0\times\mathbb{R}\to\mathrm{Z}(G)_0\times\mathbb{R}.
$$
For every subset $X\subset\mathrm{Z}(G)_0\times\mathbb{R}$ we denote by $\bbar{X}$ the connected subgroup
$$
\bbar{X}=\exp\left[\mathbb{R}\langle\log(X)\rangle\right],\quad\forall X\subset\mathrm{Z}(G)_0\times\mathbb{R}.
$$
Thus $\bbar{X}\subset G$ is a minimal Lie subgroup containing the set $X$.

\begin{proposition} Let $G$ be a simply connected almost Abelian group, $N\subset G$ a discrete normal subgroup and $H\subset G$ a connected subgroup. Then the connected subgroup $H/N\subset G/N$ is closed if and only if $\bbar{H\cap N}=H\cap\bbar{N}$.
\end{proposition}
\begin{proof} First let us note that
$$
\bbar{H\cap N}\subset H\cap\bbar{N}.
$$
Indeed, $\bbar{H\cap N}\subset\bbar{N}$ is obvious, while $\bbar{H\cap N}\subset H$ follows from $\mathbb{R}\langle\log(H\cap N)\rangle\subset\mathbf{L}_H$, where $\mathbf{L}_H$ is the Lie algebra of $H$. Let by Lemma \ref{NBLemma} $N=(H\cap N)\times B$ for a subgroup $B\subset N$. Since $N$ is a free Abelian group, we have that $\mathbb{R}\langle\log(H\cap N)\rangle\cap\mathbb{R}\langle\log(B)\rangle=0$, and because $N$ is a subgroup of the Abelian Lie group $\mathrm{Z}(G)_0\times\mathbb{R}$, it follows that $\bbar{N}=\bbar{H\cap N}\times\bbar{B}$. Thus
$$
H\cap\bbar{N}=H\cap(\bbar{H\cap N}\times\bbar{B})=\bbar{H\cap N}\times(H\cap\bbar{B}),
$$
$$
\bbar{H\cap N}=H\cap\bbar{N}\quad\Leftrightarrow\quad H\cap\bbar{B}=\{\id\}.
$$
By definition of quotient topology, $H/N\subset G/N$ is closed if and only if the complete preimage $HN\subset G$ is closed. The subgroups $H$ and $\bbar{N}$ are connected, and so is their product $H\bbar{N}$. Since $\bbar{N}\subset G$ is central, both $HN$ and $H\bbar{N}$ are subgroups. Being a connected subgroup, $H\bbar{N}\subset G$ is closed by Proposition \ref{CSubGrSCGProp}. Thus the question is reduced to whether $HN\subset H\bbar{N}$ is closed or not.

$H\cap\bbar{B}\subset\bbar{B}$ is a closed Lie subgroup, hence $\bbar{B}=H\cap\bbar{B}\times C$ where $C\subset\bbar{B}$ is a closed Lie subgroup. It follows that
$$
HN=HB=BH,\quad H\bbar{N}=H\bbar{B}=HC=CH,
$$
and we want to know whether $BH\subset CH$ is closed. Again, by definition of quotient topology, this is equivalent to $BH/H\subset CH/H$ being closed or not. Since $B\cap H=C\cap H=\{\id\}$, the homomorphisms $B\to BH/H$ and $C\to CH/H$ are isomorphisms, therefore $\rank BH/H=\rank B$ and $\dim CH/H=\dim C$, which implies that $\rank BH/H=\dim\bbar{B}\ge\dim CH/H$, and equality holds if and only if $H\cap\bbar{B}=\{\id\}$. If $H\cap\bbar{B}=\{\id\}$ then the homomorphism $\bbar{B}\to\bbar{B}H/H$ is an isomorphism, and $BH/H\subset CH/H=\bbar{B}H/H$ is closed. On the other hand, if $H\cap\bbar{B}\neq\{\id\}$ then $\dim CH/H<\rank BH/H$, therefore $BH/H\subset CH/H$ is dense (see Theorem 6.1 in \cite{StewartTall2002}). $\Box$
\end{proof}

%\end{document}

\quad

\quad

%\documentclass{article}

%\input{../include.tex}

%\begin{document}

%\cite{Freibert2012}

%\bibliographystyle{plain}
%\bibliography{../../Lib/lib}

%\end{document}

\quad
%\documentclass{article}

%\input{include.tex}

%\begin{document}

\section*{Appendix: proof of Lemma \ref{GenReduceLemma}}\label{Appendix}

Let $1<k\in\mathbb{N}$ and $(v_1,t_1),\ldots,(v_k,t_k)\in\mathbb{R}^d\times\mathbb{R}$ such that $t_i=n_it_0$, $n_i\in\mathbb{Z}$ for $i=1,\ldots,k$, where $t_0\in\mathbb{R}$.

\quad

\noindent\textbf{Lemma \ref{GenReduceLemma}}\quad\textit{There exists a change of generators $A\in\mathrm{GL}(\mathbb{Z},k)$ such that
$$
\begin{pmatrix}
u_1 & u_2 & \ldots & u_k\\
s_1 & 0 & \ldots & 0
\end{pmatrix}=\begin{pmatrix}
v_1 & v_2 & \ldots & v_k\\
t_1 & t_2 & \ldots & t_k
\end{pmatrix}\,\cdot\, A.
$$
}

\begin{remark} Here $s_1=d_*\,t_0$, where $d_*=\gcd(n_1,\ldots,n_k)$.
\end{remark}

\begin{proof} The statement amounts to the existence of an $A\in\mathrm{GL}(\mathbb{Z},k)$ such that
$$
(d_*,0,\ldots,0)=(n_1,n_2,\ldots,n_k)\cdot A.
$$
Dividing both sides by $d_*$ we reduce the problem to finding an $A\in\mathrm{GL}(\mathbb{Z},k)$ such that
\begin{equation}
(1,0,\ldots,0)=(\tilde n_1,\tilde n_2,\ldots,\tilde n_k)\cdot A,\label{AProperty}
\end{equation}
where $\tilde n_i=n_i/d_*$ for $i=1,\ldots,k$ and $\gcd(\tilde n_1,\ldots,\tilde n_k)=1$. Denote
$$
d_1\doteq\gcd(\tilde n_2,\tilde n_3,\ldots,\tilde n_k),\quad d_2\doteq\gcd(\tilde n_1,\tilde n_3,\ldots,\tilde n_k),\quad\ldots,\quad d_k\doteq\gcd(\tilde n_1,\tilde n_2,\ldots,\tilde n_{k-1}),
$$
$$
m_1\doteq\frac{\tilde n_1}{d_2d_3\ldots d_k},\quad m_2\doteq\frac{\tilde n_2}{d_1d_3\ldots d_k},\quad\ldots,\quad m_k\doteq\frac{\tilde n_k}{d_1d_2\ldots d_{k-1}},
$$
so that $\tilde n_i=m_id_1d_2\ldots d_k/d_i$ for $i=1,\ldots,k$ and $\gcd(m_i,m_j)=1$ for all $i\neq j$.

We will define the auxiliary matrix $B\in\mathrm{GL}(\mathbb{Z},k)$ depending on whether $k$ is even or odd. If $k=2r$ then define numbers $q_1,\ldots,q_k\in\mathbb{Z}$ such that, by B\'ezout's identity, $m_{2j-1}q_{2j-1}+m_{2j}q_{2j}=1$ for $j=1,\ldots,r$. Then $B$ is the following matrix,
$$
B=\begin{pmatrix}
q_1 & 0 & \ldots & 0 & -m_2 & 0 & \ldots & 0\\
q_2 & 0 & \ldots & 0 & m_1 & 0 & \ldots & 0\\
0 & q_3 & \ldots & 0 & 0 & -m_4 & \ldots & 0\\
0 & q_4 & \ldots & 0 & 0 & m_3 & \ldots & 0\\
\vdots & \vdots & \vdots & \vdots & \vdots & \vdots & \vdots & \vdots\\
0 & 0 & \ldots & q_{k-1} & 0 & 0 & \ldots & -m_k\\
0 & 0 & \ldots & q_k & 0 & 0 & \ldots & m_{k-1}
\end{pmatrix}.
$$
It is easy to see that indeed, $|\det B|=1$ and
$$
(m_1,\ldots,m_r\,|\,m_{r+1},\ldots,m_k)\cdot B=(1,\ldots,1\,|\,0,\ldots,0).
$$
If on the other hand $k=2r+3$ then we introduce the numbers $q_1,\ldots,q_{k-3}\in\mathbb{Z}$ as before, $m_{2j-1}q_{2j-1}+m_{2j}q_{2j}=1$ for $j=1,\ldots,r$. Then, again powered by B\'ezout's identity, we define integers $q_{k-2},q_{k-1},q_k,s_{k-2},s_k\in\mathbb{Z}$ such that $m_{k-2}q_{k-2}+m_{k-1}q_{k-1}+m_kq_k=1$ and $m_{k-2}s_{k_2}+m_ks_k=1$. Now the matrix $B$ is as follows,
$$
B=\begin{pmatrix}
q_1     & 0         & \ldots & 0       & 0       & -m_2     & 0         & \ldots & 0        & 0         & 0              \\
q_2     & 0         & \ldots & 0       & 0       & m_1      & 0         & \ldots & 0        & 0         & 0              \\
0       & q_3       & \ldots & 0       & 0       & 0        & -m_4      & \ldots & 0        & 0         & 0              \\
0       & q_4       & \ldots & 0       & 0       & 0        & m_3       & \ldots & 0        & 0         & 0              \\
\vdots  & \vdots    & \vdots & \vdots  & \vdots  & \vdots   & \vdots    & \vdots & \vdots   & \vdots    & \vdots         \\
0       & 0         & \ldots & q_{k-4} & 0       & 0        & 0         & \ldots & -m_{k-4} & 0         & 0              \\
0       & 0         & \ldots & q_{k-3} & 0       & 0        & 0         & \ldots & m_{k-3}  & 0         & 0              \\
0       & 0         & \ldots & 0       & q_{k-2} & 0        & 0         & \ldots & 0        & m_k       & -m_{k-1}s_{k-2}\\
0       & 0         & \ldots & 0       & q_{k-1} & 0        & 0         & \ldots & 0        & 0         & 1              \\
0       & 0         & \ldots & 0       & q_k     & 0        & 0         & \ldots & 0        & -m_{k-2}  & -m_{k-1}s_k
\end{pmatrix}.
$$
Again, it can be observed that $|\det B|=1$ and
$$
(m_1,\ldots,m_{r+1}\,|\,m_{r+2},\ldots,m_k)\cdot B=(1,\ldots,1\,|\,0,\ldots,0).
$$
For every $l\in\mathbb{N}$ denote by $C_l\in\mathrm{GL}(\mathbb{Z},l)$ the matrix
$$
C_l=\begin{pmatrix}
1 & -1 & 0 & \ldots & 0\\
0 & 1 & -1 & \ldots & 0\\
\vdots & \vdots & \vdots & \vdots & \vdots\\
0 & 0 & 0 & \ldots & 1
\end{pmatrix}.
$$
It can be easily seen that
$$
(1,\ldots,1,0,\ldots,0)\cdot\left[C_l\oplus\id_{k-l}\right]=(1,0,\ldots,0),
$$
where exactly $l$ non-zero entries are on the left-hand side. Finally, we define the auxiliary matrix $D\in\mathrm{GL}(\mathbb{Z},k)$ by $D=B\cdot\left[C_{r}\oplus\id_r\right]$ or $D=B\cdot\left[C_{r+1}\oplus\id_{r+2}\right]$ depending on whether $k=2r$ or $k=2r+3$, respectively. From what we had above it is clear that
\begin{equation}
(m_1,m_2,\ldots,m_k)\cdot D=(1,0,\ldots,0).\label{DProperty}
\end{equation}
This property of $D$ (as the more general (\ref{AProperty})) is remarkable. It means that the first column $D_{*1}$ is a B\'ezout tuple for $(m_1,\ldots,m_k)$, while the $k-1$ other columns $D_{*2},\ldots,D_{*k}$ span the hyperplane orthogonal to $(m_1,\ldots,m_k)$. It is clear that any other B\'ezout tuple for $(m_1,\ldots,m_k)$ is of the form $D\cdot(1,\lambda_2,\ldots,\lambda_k)^\top$ with $(\lambda_2,\ldots,\lambda_k)\in\mathbb{Z}^{k-1}$, and replacing the first column $D_{*1}$ in $D$ with any other such tuple will not violate (\ref{DProperty}).

Remember that $\gcd(\tilde n_1,\ldots,\tilde n_k)=1$, so that there exists a B\'ezout tuple $(p_1,\ldots,p_k)\in\mathbb{Z}^k$ such that $\tilde n_1p_1+\ldots+\tilde n_kp_k=1$. It follows that
$$
m_1\frac{d_1\ldots d_k}{d_1}p_1+\ldots+m_k\frac{d_1\ldots d_k}{d_k}p_k=1,
$$
that is, $d_1\ldots d_k\cdot(p_1/d_1,\ldots,p_k/d_k)$ is a B\'ezout tuple for $(m_1,\ldots,m_k)$, and we can afford setting $D_{*1}=d_1\ldots d_k\cdot(p_1/d_1,\ldots,p_k/d_k)$ without changing (\ref{DProperty}) or $\det D$.

The desired matrix $A$ can be constructed as below,
$$
A=\begin{pmatrix}
p_1     & D_{1,2}d_1    & \ldots    & D_{1,k}d_1    \\
\vdots  & \vdots        & \vdots    & \vdots        \\
p_k     & D_{k,2}d_k    & \ldots    & D_{k,k}d_k
\end{pmatrix}.
$$
We check that (\ref{AProperty}) is true. Indeed, $(\tilde n_1,\ldots,\tilde n_k)\cdot(p_1,\ldots,p_k)^\top=1$ by definition, whereas
$$
(\tilde n_1,\ldots,\tilde n_k)\cdot(D_{1,j}d_1,\ldots,D_{k,j}d_k)^\top=d_1\ldots d_k\cdot(m_1,\ldots,m_k)\cdot(D_{1,j},\ldots,D_{k,j})=0,\quad j=2,\ldots,k
$$
follows from (\ref{DProperty}). Finally,
$$
\det A=\left|\diag(d_1,\ldots,d_k)\cdot\begin{pmatrix}
p_1/d_1     & D_{1,2}    & \ldots    & D_{1,k}    \\
\vdots      & \vdots     & \vdots    & \vdots     \\
p_k/d_1     & D_{k,2}    & \ldots    & D_{k,k}
\end{pmatrix}\right|=
$$
$$
\left|\begin{matrix}
d_1\ldots d_kp_1/d_1     & D_{1,2}    & \ldots    & D_{1,k}    \\
\vdots      & \vdots     & \vdots    & \vdots     \\
d_1\ldots d_kp_k/d_1     & D_{k,2}    & \ldots    & D_{k,k}
\end{matrix}\right|=\det D,
$$
which proves that $A\in\mathrm{GL}(\mathbb{Z},k)$. $\Box$
\end{proof}

%\end{document}

\end{document}